\documentclass[11pt]{amsart}
\usepackage{caption}
\usepackage{longtable}
\usepackage{array}
\usepackage{multicol}
\usepackage{graphicx}
\usepackage[colorlinks=true,citecolor=black,linkcolor=black,urlcolor=blue]{hyperref}
\makeatletter
 \def\@textbottom{\vskip \z@ \@plus 1pt}
 \let\@texttop\relax
\makeatother
\usepackage{amsmath, amssymb, amsbsy, amsfonts, amsthm, latexsym, amsopn, amstext, amsxtra, euscript, amscd, color, mathrsfs}
\usepackage[normalem]{ulem}
\usepackage{soul}

\usepackage{cite}

\makeatletter

\@namedef{subjclassname@2020}{%
  \textup{2020} Mathematics Subject Classification}
\makeatother

\PassOptionsToPackage{hyphens}{url}\usepackage{hyperref}
 
 \usepackage[capbesideposition=outside,capbesidesep=quad]{floatrow}

\restylefloat{table}
\restylefloat{table}
         
\usepackage{multirow,caption}
            
\usepackage{amscd}
\usepackage{color,enumerate}

\newcommand{\RNum}[1]{\lowercase\expandafter{\romannumeral #1\relax}}

\usepackage[colorinlistoftodos,prependcaption,textsize=tiny]{todonotes}

\theoremstyle{plain}
\newtheorem{thm}{Theorem}[section]
\newtheorem{lem}[thm]{Lemma}

\newtheorem{prop}[thm]{Proposition}

\newtheorem{rmk}[thm]{Remark}

\newtheorem{thm-con}[thm]{Theorem-Conjecture}
\numberwithin{equation}{section}
\newtheorem{defn}[thm]{Definition}

\def\F{{\mathbb F}}

\begin{document}
\title[Enumeration of certain permutation group polynomials]{Enumeration of certain  permutation group polynomials}
 \author[S. U. Hasan]{Sartaj Ul Hasan}
 \address{Department of Mathematics, Indian Institute of Technology Jammu, Jammu 181221, India}
  \email{sartaj.hasan@iitjammu.ac.in}
  
  \author[R. Kaur]{Ramandeep Kaur}
  \address{Department of Mathematics, Indian Institute of Technology Jammu, Jammu 181221, India}
  \email{2022rma0027@iitjammu.ac.in}
  \author[H. Kumar]{Hridesh Kumar}
\address{Department of Mathematics, Indian Institute of Technology Jammu, Jammu 181221, India}
\email{2021rma2022@iitjammu.ac.in}
 \thanks{
 The second and third named authors are supported by the Prime Minister’s Research Fellowship, under PMRF IDs $3003658$ and $3002900$, respectively, at IIT Jammu.}
\begin{abstract}
We construct a new family of permutation group polynomials over finite fields of odd characteristic and explicitly provide its companion. Moreover, we precisely determine the number of permutation group polynomials of this form and those that are equivalent to this new family.
In addition, we completely solve the problem of enumerating permutation group polynomials of the various forms presented in Hasan and Kumar (2026), as well as permutation group polynomials which are equivalent to these families.
 \end{abstract}
\maketitle
\section{Introduction}
Let \(\mathbb{F}_q\) be the finite field with \(q\) elements, where \(q\) is a power of a prime. Moreover, we denote by $\mathbb{F}_q[X_1,X_2,\ldots,X_n]$ the ring of polynomials in $n$ variables with coefficients from $\mathbb{F}_q$.
Lagrange interpolation implies that every function 
\( f : \mathbb{F}_q^{\,n} \to \mathbb{F}_q \) admits a unique representation as a polynomial in \( n \) variables over \( \mathbb{F}_q \), where each variable has degree at most \( q-1 \). Thus, functions from \(\mathbb{F}_q^{\,n}\) to \(\mathbb{F}_q\) correspond precisely to polynomials in \(n\) variables over \(\mathbb{F}_q\) of degree less than \(q\) in each variable, and we therefore restrict our attention to polynomials over $\F_q.$
A polynomial $f(X) \in \F_q[X]$ is said to be a permutation polynomial (PP) if the induced map $c \mapsto f(c)$ is a bijection from $\F_q$ to itself. The study of permutation polynomials in one variable dates back to the classical work of Hermite in 1863~\cite{H}. Owing to their rich algebraic structure and diverse applications, the investigation of such polynomials remains an active area of research; see, for instance,~\cite{Hou_PP_15,Hou_bi_tri,JPB,OT,W,ZKZPL} and the references therein.

A polynomial $f(X_1,X_2,\ldots,X_n) \in \F_q[X_1,X_2,\ldots,X_n]$ is called a local permutation polynomial (LPP) if for each $ i \in \{1, \ldots, n\}$ and for each tuple $\overline{a}_i:=(a_1,\ldots,a_{i-1},a_{i+1},\ldots,a_n)\in \F_q^{n-1}$, the univariate polynomial $f(a_1,\ldots,a_{i-1},X_i,a_{i+1},\ldots,a_n)\in \F_q[X_i]$ is a permutation polynomial over $\F_q$. The study of local permutation polynomials was initiated by the seminal work of Mullen \cite{Mullen_1980,Mullen_1_1980}, where the author established necessary and sufficient conditions for bivariate and trivariate polynomials to be LPPs over finite fields of prime order. A system of $m$ polynomials $f_1(X_1,X_2,\ldots,X_n)$, $f_2(X_1,X_2,\ldots,X_n)$,$\ldots$, $f_m(X_1,X_2,\ldots,X_n) \in \F_q[X_1,X_2,\ldots,X_n]$, $m \leq n$, is said to be orthogonal  if the system 
\begin{equation*}
\left\{
\begin{aligned}
f_1(X_1,X_2,\ldots,X_n) &= a_1,\\
f_2(X_1,X_2,\ldots,X_n) &= a_2,\\
 &\vdots \\
f_m(X_1,X_2,\ldots,X_n) &= a_m
\end{aligned}
\right.
\end{equation*}
has exactly $q^{n-m}$ solutions in $\F_q^{n}$ for all the tuples $(a_1,a_2,\ldots,a_m) \in \F_q^{m}$. In particular, two bivariate local permutation polynomials $f_1(X_1,X_2)$ and $f_2(X_1,X_2)$ are said to be orthogonal or companion of each other if the system of equations 
\begin{equation*}
\begin{cases}
f_1(X_1,X_2)=a,\\
 f_2(X_1,X_2)=b
\end{cases}
\end{equation*}
has a unique solution $(x_1,x_2) \in \F_q \times \F_q$ for all $ (a,b) \in \F_q \times \F_q$. For a given bivariate local permutation polynomial $f_1(X_1,X_2)$, it is uncertain that whether there exists another local permutation polynomial $f_2$ such that $f_1$ and $f_2$ form an orthogonal system. This uncertainty foster an interesting research question to investigate the existence of such pairs. 

 A Latin square of order  $n$ is an $ n \times n $ array with entries from a set $S$ of size $n$, such that each element of $S$ appears exactly once in every row and column. Two Latin squares are said to be orthogonal if, when superimposed, every ordered pair from $S\times S$ appears exactly once.

Latin squares and orthogonal Latin squares play a central role in several areas, including combinatorial design theory \cite{MM_2017}, cryptography \cite{SS_1992}, and coding theory \cite{KD_2015, MGFL_2020, WW_2014}. Moreover, many fundamental and challenging open problems remain in the study of Latin squares \cite{KD_2015}. In this context, bivariate local permutation polynomials are particularly important. There exists a one-to-one correspondence between Latin squares of order of prime power $q$ and the set of bivariate LPPs over the finite fields of order $q$ \cite{Mullen_book}. Likewise, an orthogonal Latin square of the Latin square associated with a particular bivariate LPP corresponds to an orthogonal LPP, associated with that LPP.


Gutierrez and Urroz \cite{JJ_2023} recently proposed the concept of permutation group polynomials. More precisely, the authors provided a one to one correspondence between bivariate local permutation polynomials and special types of $q$-tuples of permutation polynomials in one variable. These $q$- tuples are known as permutation polynomial tuples. Specifically, the associated bivariate local permutation polynomial is called a permutation group polynomial if the elements of the $q$-tuple form a subgroup of $\mathfrak{S}_q$, the symmetric group on $\F_q$. This correspondence provides a new group theoritical approach to investigate the bivariate local permutation polynomial of special forms. In \cite{JJ_2023}, the authors introduced a family of permutation group polynomials, called $e$-Klenian polynomials, arising from a specific  subgroup of $\mathfrak{S}_q$ without fixed points. Inspired by this work, Hasan and Kumar \cite{HK}, constructed three families of permutation group polynomials over finite fields of even characterstic. Subsequently, a family of permutation group polynomials over finite fields of arbitrary characteristic was given by Hasan, Kaur, and Kumar \cite{HKK}. 
In both papers, the authors also gave explicit expressions for the companions corresponding to each family. In this paper, we provide a new family of permutation group polynomials with their companions.

Another interesting problem in the study of bivariate local permutation polynomials is the enumeration of permutation group polynomials of specific forms. Two permutation group polynomials are called of the same form if their associated groups, as subgroups of $\mathfrak{S}_q$, are isomorphic. Motivated by a non trivial problem posed by Gutierrez and Urroz \cite{JJ_2023} concerning counting of $e$-Klenian polynomials, Hasan, Kaur, and Kumar \cite{HKK} enumerated all $e$-Klenian polynomials for $e\ge1$ over $\F_q$, in other words, determined the number of permutation group polynomials of the form of $e$-Klenian polynomials over $\F_q$ for $e\ge 1$.
 The authors also provided the exact number of permutation group polynomials that are of the form those introduced in their work \cite{HKK}. In addition, the authors also computed the number of permutation group polynomials equivalent to proposed family in \cite{HKK}, and those equivalent to $e$-Klenian polynomials. Here, we enumerate all permutation group polynomials that are of the forms of permutation group polynomials constructed in this paper, as well as, permutation group polynomials constructed in \cite{HK}. We also provide the exact number of permutation group polynomials that are equivalent to these families.

The rest of the paper is structured as follows. In Section~\ref{S2}, some definitions and lemmas are introduced. In Section~\ref{S3}, we construct a new family of bivariate local permutation polynomials and determine their corresponding companions. Section~\ref{S4} is devoted to enumerate the permutation group polynomials of certain forms.

\section{Preliminaries}\label{S2}
In this section, we present several definitions and lemmas that will be useful in the following sections. The elements of $\mathbb{F}_q$ are listed throughout this work as $\F_{q} = \{c_0, c_1, \ldots, c_{q-1}\}$. If, for every $i \neq j$, the permutation $\beta_i^{-1} \beta_j$ has no fixed points, the tuple $(\beta_0, \ldots, \beta_{q-1}) \in \mathfrak{S}_q^q$ is referred to as a permutation polynomial tuple.


Due to Gutierrez and Urroz \cite{JJ_2023}, we have the following lemma which states that there is one to one correspondence between bivariate local permutation polynomials and permutation polynomial tuples over $\F_q$. 
\begin{lem}{\upshape\cite[Lemma 7]{JJ_2023}}
\label{tpl}
There is a bijection between the set of bivariate local permutation polynomials $f\in \F_q[X_1,X_2]$ and the set of $q$-tuples $\underline{\beta}_f :=(\beta_0,\ldots,\beta_{q-1}) \in {\mathfrak S}_q^q$ of permutations of $\mathbb F_q$ such that $\beta_i \beta_j^{-1}$ has no fixed points in $\F_q$ for $0 \leq i \neq j \leq q-1$. Furthermore, $f$ and $\underline{\beta}_f$ are associated to each other by the following relation: for each $0 \leq i \leq q-1$, $f(x,\beta_i(x))=c_i$ for all $x\in \F_q$.
\end{lem} 
Two bivariate local permutation polynomials $f$ and $g$ are said to be equivalent, if there exist two permutations $\lambda$ and $\sigma \in \mathfrak{S}_q$ such that $\lambda \beta_{\underline{f}}\sigma= \gamma_{\underline{g}}$, where $\beta_{\underline{f}}:=(\beta_0,\beta_1,\ldots,\beta_{q-1})$ and $\gamma_{\underline{g}}:=(\gamma_0,\gamma_1,\ldots,\gamma_{q-1})$ are corresponding permutation polynomials tuples of $f$ and $g$, respectively. A bivariate local permutation polynomials $f$ is said to be a permutation group polynomial if the elements of the corresponding permutation polynomial tuple $\beta_{\underline{f}}:=(\beta_0,\beta_1,\ldots,\beta_{q-1})$ form a subgroup of $\mathfrak{S}_q$. 
\begin{defn}\label{D21}{\upshape \cite[Definition 4.1]{HK}}
 Let $h_1,h_2\in\F_q[X]$ be two permutation polynomials. Let $A :=\{(c,h_1(c)):c \in \F_q\}$ and $B:=\{ (c,h_2(c)):c \in \F_q\}$. We say that $h_1$ intersects $h_2$ simply  if $|A \cap B|=1$. 
 \end{defn}
 
 \begin{defn}\label{D22}{\upshape\cite[Definition 4.2]{HK}}
Let $f\in\F_q[X,Y]$ be an LPP and $\underline{\beta}_f=(\beta_0,\ldots,\beta_{q-1}) \in {\mathfrak S}_q^q$ be the corresponding permutation polynomial tuple. We say that a univariate permutation polynomial $h$ intersects the bivariate LPP $f$ simply if $h$ intersects $\beta_i$ simply for each $0 \leq i \leq q-1$. 
 \end{defn}
 
 \begin{lem}\label{L21}{\upshape\cite[Lemma 4.4]{HK}}
Let $f\in \F_q[X,Y]$ be a permutation group polynomial and $h\in\F_q[X]$ be a permutation polynomial which intersects $f$ simply. Let $\underline{\beta}_f=(\beta_0,\ldots,\beta_{q-1})$ be the permutation polynomial tuple corresponding  to $f$. Then the polynomial $g$ associated  to $\underline{\gamma}_g=(h \beta_0, \ldots,h \beta_{q-1})$ is a companion of $f$.
\end{lem}

In Section \ref{S4}, we enumerate all the permutation group polynomials, which are equivalent to the permutation group polynomials corresponding to the following three subgroups of $\mathfrak{S}_q$ constructed in \cite{HK}.

  \begin{lem}\label{HK1}{\upshape \cite[Theorem 3.1]{HK}}
  Let $q=2^m$, where $m\geq 3$ is a positive integer. Also, let
\[
a=(c_0,c_1)(c_2,c_3)\cdots(c_{q-2},c_{q-1}),
\]
and
\begin{equation*}
    \begin{split}
        b= &(c_{\frac{q}{2}-2},\ldots,c_{\frac{q}{2}-2(n_1+1)}, \ldots,c_{2},c_{0},  c_{\frac{q}{2}},\ldots, c_{\frac{q}{2}+2n_1},\ldots,c_{q-4},c_{q-2})\\
        &(c_{1},c_{3},\ldots, c_{2n_1+1}, \ldots,c_{\frac{q}{2}-1},c_{q-1},c_{q-3},\ldots, c_{q-(2n_1+1)} \ldots,c_{\frac{q}{2}+1}),
    \end{split}
\end{equation*}
 be permutations of $\F_q$, where $0 \leq n_1 \leq \frac{q-4}{4}$ is an integer. Then $G=\langle a,b: |a|=2, |b|=\frac{q}{2}, \text { and } ab=b^{-1}a \rangle=\{b^{j}a^{i}: 0 \leq j \leq \frac{q}{2}-1,0\leq i \leq 1, i,j \in \mathbb{Z} \}$ is a subgroup of ${\mathfrak S}_q$ of order $q$ and none of its elements, except the identity, has a fixed point. 
  \end{lem}
 \begin{lem}\label{HK2}{\upshape \cite[Theorem 3.9]{HK}}
Let $q=4k$ be a positive integer, where $k=2^{\ell}$ for some positive integer $\ell \geq 2$. Also, let
\[
a=(c_0,c_1)(c_2,c_3)\cdots(c_{4k-2},c_{4k-1}),
\]
and
\begin{equation*}
    \begin{split}
        b=& (c_0, c_{2k-2}, \ldots, c_{2n_1}, c_{2(2n_1+k-1)}, \ldots , c_{k-2}, c_{4k-6},c_k,c_{2k}, \ldots, c_{2n_2+k},c_{2(2n_2+k)}, \ldots, c_{2k-4}, c_{4k-8}, \\
        &  c_{4k-1}, c_{4k-4}) (c_1, c_{2k+1}, \ldots, c_{2n_1+1}, c_{2(2n_1+k)+1}, \ldots , c_{k-1}, c_{4k-3},c_{k+1},c_{2k-1}, \ldots, c_{2n_2+k+1},\\
        & c_{2(2n_2+k)-1}, \ldots, c_{2k-3}, c_{4k-9},  c_{4k-2}, c_{4k-5}),
    \end{split}
\end{equation*}
where $n_1 \in \{0,1,\ldots, \frac{k-2}{2}\}$ and $n_2 \in \{0,1,\ldots, \frac{k-4}{2}\}$.
Then $G=\langle a,b: |a|=2, |b|=2k \text{ and } ab=b^{k+1}a\rangle=\{b^{j}a^{i}: 0 \leq j \leq 2k-1=\frac{q}{2}-1, 0\leq i \leq 1\}$ is a subgroup of $\mathfrak {S}_q$ in which  no permutation has fixed points except the identity permutation and $|G|=q$.
\end{lem}
\begin{lem}\label{HK3}{\upshape \cite[Theorem 3.13]{HK}}
Let $q=4k$ be a positive integer, where $k=2^{\ell}$ for some positive integer $\ell \geq 2$. Also, let 
\[
a=(c_0,c_1)(c_2,c_3)\cdots(c_{q-2},c_{q-1})
\]
and
\begin{equation*}
    \begin{split}
        b=& (c_0, c_{2k+1}, \ldots, c_{2n_1}, c_{2n_1+2k+1}, \ldots , c_{k-2}, c_{3k-1}, c_{k+1},c_{3k},c_{k+2},c_{3k+2}, \ldots, c_{2(n_2+1)+k},c_{2(n_2+1)+3k},\\
        & \ldots, c_{2k-2}, c_{4k-2}) (c_{3k-2},c_{2k-1}, \ldots, c_{3k-2-2n_2}, c_{2k-1-2n_2}, \ldots ,c_{2k+2},c_{k+3}, c_{2k},c_{k}, c_{4k-1}, c_{k-1}, \\
        & \ldots, c_{4k-1-2n_1},c_{k-1-2n_1}, \ldots, c_{3k+1},c_{1}), 
    \end{split}
\end{equation*}
where $n_1 \in \{0,1,\ldots,\frac{k-2}{2}\}$ and $n_2 \in \{0,1,\ldots,\frac{k-4}{2}\}$, be two permutations of ${\mathfrak S}_q$.
Then the set $G=\langle a,b : |a|=2, |b|=2k \text{ and } ab=b^{k-1}a\rangle =\{b^{j}a^{i}: 0 \leq j \leq 2k-1=\frac{q}{2}-1, 0\leq i \leq 1\}$ is a subgroup of ${\mathfrak S}_q$ of order $q$ in which no permutation has any fixed point except the identity permutation.
\end{lem}
\begin{lem}{\upshape \cite[Theorem 4.5]{HKK}}\label{Equivalent_PGP}
Let $f$ be a permutation group polynomial and its corresponding permutation polynomial tuple is $\underline{\beta}_f=(\beta_0,\beta_1,\ldots,\beta_{q-1})\in \mathfrak{S}_q^{q}$. Then there are $\frac{q(q!)}{|\mathfrak{C}_{\mathfrak{S}_q}(G)|}$ permutation group polynomials that are equivalent to $f$, where $G=\{\beta_0,\beta_1,\ldots,\beta_{q-1}\}$ and $\mathfrak{C}_{\mathfrak{S}_q}(G)$ is the centralizer of $G$ in $\mathfrak{S}_q$.
\end{lem}

 \section{Permutation group polynomials and companions}\label{S3}
This section presents a family of permutation group polynomials over finite fields of odd characteristic, together with their corresponding companions. Before establishing the main results, we first prove the following two lemmas.
 
 \begin{lem}\label{RR1}
 Let $q=p^3$, where $p$ is an odd prime. Let $\F_q = \{ c_0, c_1, \ldots, c_{q-1}\}$ be the finite field with $q$ elements,
\[
a=(c_{0},c_{1},\ldots,c_{p^2-1})(c_{p^2},c_{p^2+1},\ldots,c_{2p^2-1})\cdots (c_{p^2(p-1)},c_{p^2(p-1)+1},\ldots,c_{p^3-1}), 
\]
 and 
\[
b=b_{0}b_{1}^{1+p}b_2^{1+2p}\cdots b_{p-1}^{1+(p-1)p}
\]
be permutations of $\F_q$, where 
\begin{equation*}
\begin{split}
b_{k}=(c_{k},c_{p^2+k},\ldots,c_{(p-1)p^2+k},c_{p+k},c_{p+p^2+k},\ldots,c_{p+(p-1)p^2+k},\ldots,c_{(p-1)p+k},c_{(p-1)p+p^2+k},\ldots,& \\ c_{(p-1)p+(p-1)p^2+k}). 
\end{split}
\end{equation*}
Then $ab=b^{p^2-p+1}a$.
 \end{lem}
 \begin{proof}
 Note that for any $t \in \{0,1,\ldots,q-1=p^3-1\}$, there exists a unique positive integer $r_t$ depending solely on $t$ satisfying $ (r_t-1)p^{2} \leq t \leq r_tp^{2}-1$. Let $c_t$ be an element of $\F_q$, where $t \in \{0,1,\ldots,q-1\}$. Then by the cyclic structure of $a$, we have
 \begin{equation}\label{e21}
a(c_{t})=\begin{cases}

       c_{t+1} & \text{ if } t < r_tp^{2}-1, \\
       c_{t+1-p^2} & \text{ if } t=r_tp^{2}-1
       	\end{cases}
=c_{(r_t-1)p^{2}+(t+1)\pmod {p^{2}}}.
\end{equation}
For any $k \in \{0,1,\ldots,p-1\}$, the permutation $b_k$ acts on $c_t$ as follows 
\begin{equation}\label{e22}
b_k(c_{t})=\begin{cases}
       c_{t} & \text{ if } t \not \equiv k  \pmod  p, \\
       c_{t+p^2} & \text{ if } t < p^{3}-p^2 \text{ and } t \equiv k  \pmod  p, \\
       c_{t-p^3+p^2+p} & \text{ if } p^{3}-p^2 \leq  t < p^{3}-p \text{ and } t \equiv k  \pmod  p,\\
       c_{t-p^3+p} & \text{ if }    t \geq p^{3}-p \text{ and } t \equiv k \pmod  p.
       	\end{cases}
\end{equation}
To prove the identity $ab=b^{p^2-p+1}a$, we first show that $$ab_i=b_{i+1}a$$ for all $ i \in \{0,1,\ldots,p-2\}$ and $ab_{p-1}=b_0a$. To prove that $ab_i=b_{i+1}a$ for any $i \in \{0,1,\ldots,p-2\}$, we consider the following two cases.

\textbf{Case 1:} Let $t \not \equiv i \pmod  p$. Then by using Equations \eqref{e21} and \eqref{e22}, we have
\[
ab_i(c_t)=a(c_t)=c_{(r_t-1)p^2+(t+1) \pmod {p^2}}
\]
and
\[
b_{i+1}a(c_t)=b_{i+1}(c_{(r_t-1)p^2+(t+1) \pmod {p^2}})=c_{(r_t-1)p^2+(t+1) \pmod {p^2}}
\]
as $((r_t-1)p^2+(t+1) \pmod {p^2} )\not \equiv (i+1) \pmod p$ for  $t \not \equiv i  \pmod p$. 

\textbf{Case 2:} Suppose that $t \equiv i \pmod  p$. It is trivial that $t \not \equiv -1 \pmod  p$ for $i \in \{0,1,\ldots,p-2\}$. Since $t \not \equiv -1 \pmod  p$, $ (r_t-1)p^{2} \leq t < r_tp^{2}-1$ and so $a(c_t)=c_{t+1}$.  We further split this case into three subcases:

\textbf{Subcase 2.1:} Let $t < p^{3}-p^2$. Since $t \not \equiv -1  \pmod  p$, it follows that $t \leq p^{3}-p^2-2$. By using Equations \eqref{e21} and \eqref{e22},
\[
ab_i(c_t)=a(c_{t+p^2})=c_{t+p^2+1} 
\]
as  $t+p^2 \not \equiv -1 \pmod  p.$ Next,
\[
b_{i+1}a(c_t)=b_{i+1}(c_{t+1})=c_{t+1+p^2},
\]
since $t+1 \equiv i+1  \pmod  p$ and  $t+1<p^{3}-p^2$. Thus, $ab_i(c_t)=b_{i+1}a(c_t)$.

\textbf{Subcase 2.2:} Consider that $p^{3}-p^2 \leq t < p^{3}-p$. Again $t \not \equiv -1  \pmod  p$ implies that $t \leq p^{3}-p-2$. Therefore, we have
\[
ab_i(c_t)=a(c_{t-p^3+p^2+p})=c_{t-p^3+p^2+p+1} \text{ as } t-p^3+p^2+p \not \equiv -1 \pmod  p.
\]
Also,
\[
b_{i+1}a(c_t)=b_{i+1}(c_{t+1})=c_{t-p^3+p^2+p+1}=ab_i(c_t),
\]
as $t+1 \equiv i+1  \pmod  p$ and  $p^{3}-p^2 < t+1<p^{3}-p$.

\textbf{Subcase 2.3:} Let $p^{3}-p \leq t < p^{3}-1$.  In this case, we obtain
\[
ab_i(c_t)=a(c_{t-p^3+p})=c_{t-p^3+p+1} \text{ as } t-p^3+p \not \equiv -1 \pmod  p.
\]
Moreover,
\[
b_{i+1}a(c_t)=b_{i+1}(c_{t+1})=c_{t-p^3+p+1}=ab_i(c_t),
\]
as $t+1 \equiv i+1 \pmod  p$ and  $p^{3}-p < t+1 \leq p^{3}-1$.

Considering both cases discussed above, we conclude that $ab_i=b_{i+1}a$ for $i \in \{0,1,\ldots,p-2\}$. Now, we shall show that $ab_{p-1}=b_0a$. If $t \not \equiv -1  \pmod  p$, then we can easily see that $ab_{p-1}(c_t)=b_0a(c_t)$ by an argument similar to Case 1. For $t  \equiv -1 \pmod  p$, due to the definition of $b_{p-1}$, we proceed by considering three natural cases, namely, $0 \leq t < p^{3}-p^2$, $p^3-p^2 \leq t < p^{3}-p$ and $p^{3}-p \leq t \leq p^{3}-1$. Here, we prove the identity $ab_{p-1}=b_0a$ only for the elements
$c_t$ with $t  \equiv -1 \pmod  p$ and $0 \leq t < p^{3}-p^2$. The same reasoning can be applied to the remaining elements of $\F_q.$


First, let $0 \leq t < p^{3}-p^2$, $t \equiv -1 \pmod  p$ and $t  \not \equiv -1 \pmod {p^2}$.  As $0 \leq t < p^{3}-p^2$, thus using the definition of $b_{p-1}$ (see Equation \eqref{e22} ), we have
\[
ab_{p-1}(c_t)=a(c_{t+p^2}).
\]
Since $t+p^2 \not \equiv -1  \pmod {p^2}$, we further have $ab_{p-1}(c_t)=a(c_{t+p^2})=c_{t+p^2+1}$.
On the other hand, $b_0a(c_t)=b_0(c_{t+1})$ as $t \not \equiv -1  \pmod {p^2}$.  Since  $t < p^{3}-p^2$ and $t \not \equiv -1  \pmod {p^2}$, it follows that $t+1 < p^{3}-p^2$ and thus $b_0a(c_t)=b_0(c_{t+1})=c_{t+1+p^2}$.

Next, we assume that $0 \leq t < p^{3}-p^2$, $t \equiv -1 \pmod p$ and $t \equiv -1 \pmod {p^2}$. Then using the definitions of $a$ and $b_{p-1}$ from Equations \eqref{e21} and \eqref{e22}, we obtain
\[
ab_{p-1}(c_t)=a(c_{t+p^2})=c_{t+p^2+1-p^2}=c_{t+1}, \text{ and}
\]
\[
b_0a(c_t)=b_0(c_{t+1-p^2})=c_{t+1-p^2+p^2}=c_{t+1}=ab_{p-1}(c_t), \text{ as } t+1-p^2 < p^{3}-p^2.
\]


Consequently, we have  $ab_{p-1}=b_0a$ and $ab_i=b_{i+1}a$ for $i \in \{0,1,\ldots,p-2\}$. Hence, 
\begin{equation*}
\begin{split}
ab&=ab_{0}b_{1}^{1+p}b_2^{1+2p}\cdots b_{p-1}^{1+(p-1)p}
\\& =b_1ab_{1}^{1+p}b_2^{1+2p}\cdots b_{p-1}^{1+(p-1)p}
\\&= b_1b_{2}^{1+p}ab_2^{1+2p}\cdots b_{p-1}^{1+(p-1)p}
\\&
\hspace{1.5cm}\vdots
\\&=b_1b_{2}^{1+p}b_3^{1+2p}\cdots b_{p-1}^{1+(p-2)p}b_{0}^{1+(p-1)p}a
\\&=b_{0}^{1+(p-1)p}b_1b_{2}^{1+p}b_3^{1+2p}\cdots b_{p-1}^{1+(p-2)p}a
\\&=b^{p^2-p+1}a
\end{split}
\end{equation*}
as $b_{i}$'s are disjoint cycles. This completes the proof.
 \end{proof}
 \begin{lem}\label{RR2}
 Let $q=p^3$ and $\F_q=\{c_0,c_1,\ldots,c_{q-1}\}$ be the finite field, where $p$ is an odd prime. Moreover, let $a$ and $b$ be as defined in Lemma \ref{RR1}. Then we have $a^p=b^p$. 
 \end{lem}
 \begin{proof}
  Using the definition of $b$ from Lemma \ref{RR1}, we have $b^p=b_{0}^pb_{1}^{p}b_2^{p}\cdots b_{p-1}^{p}$ as $b_i$'s are disjoint cycles. Let  $ t\in \{0,1,\ldots,q-1\}$ and $t \equiv k \pmod  p$ for some $k \in \{0,1,\ldots,p-1\}$.  Then, by Equation \eqref{e22}, we observe that $b^p(c_t)=b_k^p(c_t)$, since $b_i$'s are disjoint cycles and $b_i(c_t)=c_t$ whenever $t \not\equiv i \pmod  p$. Thus, it suffices to show that $a^p(c_t)=b_k^p(c_t)$, where $t \equiv k \pmod  p$. Before proving the result, first we shall show that the definition of $a$ from \eqref{e21} implies that
\[
a^{d}(c_t)=c_{(r_t-1)p^2+(t+d) \pmod {p^2}}
\]
for any integer $d \geq 0$ and $t \in \{0,1,\ldots,p^3-1\}$ such that $(r_t-1)p^2\leq t \leq r_tp^2-1$.
We prove this by applying induction on $d$. For $d=0$, the result is true as $t \pmod {p^2}+(r_t-1)p^2=t$. Now assume that result is true for $d=d'$, that is, 
\[
a^{d'}(c_t)=c_{(r_t-1)p^2+(t+d') \pmod {p^2}}.
\]
Since $a^{d'+1}(c_t)=a(a^{d'}(c_t))=a(c_{(r_t-1)p^2+(t+d') \pmod {p^2}})$ and $(r_t-1)p^2 \leq (r_t-1)p^2+(t+d') \pmod {p^2} \leq r_tp^2-1$,  Equation \eqref{e21} implies 
\[
a^{d'+1}(c_t)=c_{(r_t-1)p^2+((r_t-1)p^2+(t+d') \pmod {p^2}+1) \pmod {p^2}}=c_{(r_t-1)p^2+(t+d'+1)\pmod {p^2}}.
\]
Thus by induction,  
\begin{equation}\label{tempe4}
a^{d}(c_t)=c_{(r_t-1)p^2+(t+d) \pmod {p^2}}.
\end{equation}
In particular
\[
a^p(c_t)=c_{(r_t-1)p^2+(t+p)  \pmod {p^2}}=\begin{cases}

       c_{t+p-p^2} & \text{ if } r_tp^{2}-p \leq t \leq r_tp^{2}-1,
       \\
       c_{t+p} & \text{ if } (r_t-1)p^{2} \leq t < r_tp^{2}-p. 
       	\end{cases}
\]
This leads us naturally to the following two cases.

\textbf{Case 1:} 
We first assume that $r_t p^{2} - p \leq t \leq r_t p^{2} - 1$. Since $r_t p^{2} - p \leq t \leq r_t p^{2} - 1$, we can express $t = r_t p^{2} - p + x,$ where  $0 \leq x \leq p-1$.
This yields
\[
t \equiv (p-1)p + x \pmod {p^{2}}.
\]
We now show that either $t < p^{3} - p^{2}$ or $t \geq p^{3} - p$. Suppose, for the sake of contradiction, that $p^{3} - p^{2} \leq t < p^{3} - p$. 
Then $t$ can be written as 
\[
t = p^{3} - p^{2} + y, \quad \text{where } 0 \leq y < (p-1)p.
\]
In this case, $t \equiv y \pmod{p^{2}},
$
where $y < (p-1)p$. Thus, $(p-1)p + x \equiv y \pmod {p^2}$. Since $0 \leq (p-1)p + x \leq p^{2} - 1$, $0 \leq y <p^2-p<p^2-1$, therefore, it implies that $y= (p-1)p + x$. This is a contradiction. Hence, we must have either $t < p^{3} - p^{2}$ or $t \geq p^{3} - p$. 
Accordingly, we divide this case into the two subcases namely, $t < p^3-p^2$ and $t \geq p^3-p$.


\textbf{Subcase 1.1:} Let $t < p^3-p^2$. We also have $ t \geq r_tp^{2}-p $, which implies $r_t <p$. Now for any fix integer  $1 \leq i \leq p-r_t$ and by repeatedly using the definition of $b_k$ from \eqref{e22}, we have
\[
b_{k}^{i}(c_t)=c_{t+ip^2},
\]
as $t+(i-1)p^2 \leq r_tp^2-1+(p-r_t-1)p^2=p^3-p^2-1< p^3-p^2$. In particular, $b_{k}^{p-r_t}(c_t)=c_{t+(p-r_t)p^2}$ and also $$t+(p-r_t)p^2 \geq r_tp^{2}-p+(p-r_t)p^2=p^3-p,$$ therefore, 
$$b_{k}^{p-r_t+1}(c_t)=b_{k}(c_{t+(p-r_t)p^2})=c_{t+(p-r_t)p^2-p^3+p}.$$
Since $b_{k}^{p}(c_t)=b_{k}^{r_t-1}(b_{k}^{p-r_t+1}(c_t))=b_{k}^{r_t-1}(c_{t+(p-r_t)p^2-p^3+p})$, so again we can see that for any positive integer $1 \leq j \leq r_t-1$,
\begin{equation*}
\begin{split}
(j-1)p^2+t+(p-r_t)p^2-p^3+p & \leq  (r_t-2)p^2+r_tp^2-1+(p-r_t)p^2-p^3+p \\&=(r_t-2)p^2-1+p \\&<p^3-2p^2-1+p\\&<p^3-p^2-1
\end{split}
\end{equation*}
as $r_t<p$.
Thus for any fix integer $1 \leq j \leq r_t-1$ and using Equation \eqref{e22}, we get 
\[
b_{k}^{j}(c_{t+(p-r_t)p^2-p^3+p})=c_{jp^2+t+(p-r_t)p^2-p^3+p}. 
\]
 In particular, 
 $$b_{k}^{r_t-1}(c_{t+(p-r_t)p^2-p^3+p})=c_{(r_t-1)p^2+t+(p-r_t)p^2-p^3+p}.$$ 
 Finally, we have
\[
b_{k}^{p}(c_t)=b_{k}^{r_t-1}(b_{k}^{p-r_t+1}(c_t))=b_{k}^{r_t-1}(c_{t+(p-r_t)p^2-p^3+p})=c_{(r_t-1)p^2+t+(p-r_t)p^2-p^3+p}=c_{t+p-p^2}=a^{p}(c_t).
\]

\textbf{Subcase 1.2:} Let $t \geq p^3-p$. In this case,
\[
b_{k}(c_t)=c_{t-p^3+p}.
\]
We also have $t \leq p^3-1$, which implies $t-p^3+p \leq p-1$. Since for each integer $1 \leq i \leq p-1$, we have
\begin{equation*}
\begin{split}
t-p^3+p+(i-1)p^2 \leq p^3-1-p^3+p+(p-2)p^2=p^3-1+p-2p^2<p^3-p^2-1,
\end{split}
\end{equation*}
therefore,
\[
b_{k}^{i}(c_{t-p^3+p})=c_{t-p^3+p+ip^2}.
\]
From the above equality, we get $b_{k}^{p-1}(c_{t-p^3+p})=c_{t-p^3+p+(p-1)p^2}$. Hence,
 \[
 b_{k}^{p}(c_t)=b_k^{p-1}b_k(c_t)=b_k^{p-1}(c_{t-p^3+p})=c_{t-p^3+p+(p-1)p^2}=c_{t+p-p^2}=a^{p}(c_t).
 \]

\textbf{Case 2:} We now consider that $(r_t-1)p^2 \leq t < r_tp^{2}-p$. First, we shall show that either $t<p^3-p^2$ or $ p^3-p^2 \leq t <p^3-p$. In this case, from a similar argument as in Case 1, we have $t \equiv x \pmod {p^2}$, where $0 \leq x <p(p-1)$. If possible, suppose that  $p^3-p \leq t \leq p^3-1$, then $t \equiv p(p-1)+y \pmod {p^2}$ for some $0 \leq y \leq p-1$. This implies that $p(p-1)+y \equiv x \pmod {p^2}$, which is a contradiction as $0 \leq p(p-1)+y \leq p^2-1$, $0 \leq x <p^2-p<p^2-1$ and $ x\neq p(p-1)+y$. Therefore, $t<p^3-p^2$ or $ p^3-p^2 \leq t <p^3-p$. Again, we divide this case into the following two subcases.

\textbf{Subcase 2.1:} Let $t<p^3-p^2$. We also have $(r_t-1)p^2 \leq t$, which gives $r_t<p$. Similarly as in Subcase 1.1, we have
\[
b_{k}^{p-r_t}(c_t)=c_{t+(p-r_t)p^2}.
\]
Note that $t+(p-r_t)p^2 <r_tp^2-p+(p-r_t)p^2=p^3-p$ and $t+(p-r_t)p^2 \geq (r_t-1)p^2+(p-r_t)p^2=p^3-p^2$. Therefore, we get
\[
b_{k}^{p-r_t+1}(c_t)=b_{k}(c_{t+(p-r_t)p^2})=c_{t+(p-r_t)p^2-p^3+p^2+p}.
\]
Now for any fix integer $1 \leq j \leq r_t-1$, we have 
\begin{equation*}
\begin{split}
t+(p-r_t)p^2-p^3+p^2+p+(j-1)p^2&<r_tp^2-p+(p-r_t)p^2-p^3+p^2+p+(r_t-2)p^2\\&=r_tp^2-p^2<p^3-p^2
\end{split}
\end{equation*}
 as $r_t<p$. Using the similar argument as in Subcase 1.1, we have 
 $$b_{k}^{r_t-1}(c_{t+(p-r_t)p^2-p^3+p^2+p})=c_{t+(p-r_t)p^2-p^3+p^2+p+(r_t-1)p^2},$$ therefore,
\[
b_{k}^{p}(c_t)=b_{k}^{r_t-1}(b_{k}^{p-r_t+1}(c_t))=b_{k}^{r_t-1}(c_{t+(p-r_t)p^2-p^3+p^2+p})=c_{t+(p-r_t)p^2-p^3+p^2+p+(r_t-1)p^2}=c_{t+p}=a^{p}(c_t).
\]

\textbf{Subcase 2.2:} Let $ p^3-p^2 \leq t <p^3-p$. From Equation \eqref{e22}, we get
\[
b_{k}(c_t)=c_{t-p^3+p^2+p}.
\]
Using the similar argument as in Subcase 1.2, we get $b_{k}^{p-1}(c_{t-p^3+p^2+p})=c_{t-p^3+p^2+p+(p-1)p^2}$, hence,
\[
b_{k}^{p}(c_t)=b_{k}^{p-1}(c_{t-p^3+p^2+p})=c_{t-p^3+p^2+p+(p-1)p^2}=a^{p}(c_t).
\]
Thus above two cases conclude that $a^p=b^p$.
 \end{proof}
\begin{thm}\label{T31}
Let $q=p^3$, where $p$ is an odd prime, and let $\F_q=\{c_0,c_1,\ldots,c_{q-1}\}$ denote the finite field with $q$ elements. Suppose that $a$ and $b$ defined as in Lemma~\ref{RR1} are permutations of $\F_q$. Then $G=\langle a,b\rangle$ is a subgroup of $\mathfrak{S}_q$ of order $q$, and every non-identity element of $G$ has no fixed point in $\F_q$.
\end{thm}
\begin{proof}
First, we shall show that $|G|=q=p^3$. By the definitions of $a$ and $b$, it is clear that $|a|=|b|=p^2$. From Lemma \ref{RR1}, we have $ab=b^{p^2-p+1}a$ and thus $G=HK$ , where $H=\langle b \rangle$ and $K=\langle a \rangle $ are subgroups of $G$. Therefore,
\[
|G|=|HK|=\dfrac{|H||K|}{|H \cap K|}=\dfrac{p^4}{|H \cap K|}.
\]
It is clear that $|H \cap K|=1,p$ or $p^2$ as $H \cap K \leq H,K$ and $|H|=|K|=p^2$. From Lemma \ref{RR2}, we also have $a^p=b^p$, which gives $|H \cap K| = p$ or $p^2$. If $|H \cap K| = p^2$, then $H=K$. It gives $a=b^u$ for some
$1 \leq u \leq p^2-1$. As $b^u(c_t)=c_{t'}$ for some $t' \in \{0,1,\ldots,q-1\}$ and Equation \eqref{e22} implies that $t'\equiv t \pmod  p$. Therefore, $a(c_t)=c_{t'}$ such that $t'\equiv t \pmod  p$. This is not possible from the definition of $a$ (see Equation \eqref{e21}). Thus, $|H \cap K| = p$ and $|G|=p^3$. 


Now, it only remains to show that every non-identity element of $G$ has no fixed point in $\F_q$. Since $G=HK$ and $a^p=b^p$, any element of $G$ can be written as $b^{j}a^{i}$ for some $j\in \{0,1,\ldots,p^2-1\}$ and $i \in \{0,1,\ldots,p-1\}$. For any fix $j\in \{1,2,\ldots,p^2-1\}$, $b^j$ has no fixed point as $b$ is the product of $p$ disjoint cycles each of length $p^2$. Therefore, it is sufficient to show that $b^{j}a^{i}$ does not have a fixed point for all $j\in \{0,1,\ldots,p^2-1\}$ and $i \in \{1,2,\ldots,p-1\}$. Let $1 \leq i_1 \leq p-1$, $0 \leq j_1 \leq p^2-1$ be two integers and $t \in \{0,1,\ldots,q-1\}$ such that $(r_t-1)p^2 \leq t \leq r_tp^2-1$ for some integer $1 \leq r_t \leq p$. Then using the definition of $a$ from Equation \eqref{tempe4}, we obtain 
\[
b^{j_1}a^{i_1}(c_t)=b^{j_1}(c_{(r_t-1)p^2+(t+i_1) \pmod {p^2}}).
\]
 Let $t'=(r_t-1)p^2+(t+i_1) \pmod {p^2}$ and $b^{j_1}(c_{t'})=c_{t''}$ for some $t'' \in \{0,1,\ldots,q-1\}$. Then, by the definition of $b$, we must have $t'' \equiv t' \pmod  p$. If possible, suppose that $b^{j_1}a^{i_1}(c_t)=c_{t''}=c_t$, i.e., $t''=t$. Then $t \equiv t'' \pmod  p$. Using the above equivalence, we have $t \equiv t' \equiv t+i_1 \pmod  p$, which implies $i_1 \equiv 0 \pmod  p$. This is a contradiction, as $1 \leq i_1 \leq p-1$. Thus, every non-identity element of $G$ has no fixed point in $\F_q$, which completes the proof.
\end{proof}

Now we show that the family of permutation group polynomials constructed in Theorem \ref{T31} is inequivalent to all known families of permutation group polynomials. To the best of our knowledge, the existing families appear in \cite{JJ_2023, HK, HKK}. In \cite[Theorem 3.18]{HK}, it was established that if two permutation group polynomials are equivalent, then their corresponding groups are conjugate to each other in $\mathfrak{S}_q$.
\begin{prop}
Let $q = p^{3}$, where $p$ is an odd prime, and let $f$ be the permutation group polynomial corresponding to the group constructed in Theorem~\ref{T31}. Then $f$ is not equivalent to any of the known families of permutation group polynomials.
\end{prop}
\begin{proof}
The proof follows from the fact that all subgroups corresponding to the permutation group polynomials in \cite{JJ_2023,HKK} are Abelian, while the permutation group polynomials in \cite{HK} are defined over finite fields of even characteristic. In contrast, the subgroup constructed in Theorem~\ref{T31} is non-Abelian, and the associated permutation group polynomials are defined over finite fields of odd characteristic.
\end{proof}
Following theorem provides the companions for the permutation group polynomials corresponding to the subgroup defined in Theorem \ref{T31}.
\begin{thm}\label{T34}
 Let $q=p^3$, where $p$ is an odd prime, and $f\in \F_q[X_1,X_2]$ be a permutation group polynomial corresponding to the group $G$ constructed in Theorem \ref{T31}.  Let $\underline{\beta}_{f} = (\beta_0, \beta_1, \ldots, \beta_{q-1})$ be the permutation polynomial tuple corresponding to $f$, where $\beta_{t} = a^{k} b^{ip+j}$ for $
t = ((i+jk)\bmod p)p + jp^{2} + k$ with $i,j,k \in \{0,1,\ldots,p-1\}$. Then the bivariate local permutation polynomial $g(X_1,X_2)$ corresponding to the permutation polynomial tuple $(h\beta_0, h\beta_1, \ldots, h\beta_{q-1})$ is a companion of $f$, where $h$ is defined as follows
\[
h(c_{\left((i+jk)\pmod  p\right)p+jp^2+k}):=a^{k}b^{ip+j}(c_{\left((i+jk)\pmod  p\right)p+jp^2+k}).
\] \end{thm}
\begin{proof}
To prove  $g(X_1,X_2)$ is a companion of $f$, from Lemma \ref{L21}, it suffices to show that $h(X) \in \F_q[X]$ is a permutation and it simply intersects $f$. Before proving this, we shall show that $$\left\{0,1,\ldots,p^3-1\right\}=\left\{\left((i+jk)\pmod  p\right)p+jp^2+k \mid 0 \leq i,j,k \leq p-1\right\}:=S.$$ We can easily see that $S\subseteq \{0,1,\ldots,p^3-1\}$ as $0 \leq i,j,k \leq p-1.$ Now if possible suppose that there exist $(i_1, j_1, k_1) \neq (i_2, j_2, k_2)$ such that 
$$\left((i_1+j_1k_1)\pmod  p\right)p+j_1p^2+k_1=\left((i_2+j_2k)\pmod  p\right)p+j_2p^2+k_2.$$ 
From the above equation, we have $k_1 \equiv k_2 \pmod  p$, which gives $k_1=k_2$. If $j_1 \neq j_2$, say $j_1 < j_2$ without loss of generality, then 
$$\left((i_1+j_1k_1)\pmod  p\right)p+j_1p^2+k_1<\left((i_2+j_2k)\pmod  p\right)p+j_2p^2+k_2,$$ which is a contradiction.
 Therefore, we also have $j_1=j_2$, which implies that $i_1=i_2$. Thus $|S|=p^3$ and $S=\{0,1,\ldots,p^3-1\}$. This implies that $\{\beta_0, \beta_1, \ldots, \beta_{q-1}\}=G$, where $G=\langle a,b \mid a^p=b^p, |a|=|b|=p^2, \text{ and } ab=b^{p^2-p+1}a \text{ or } ab^{p+1}=ba\rangle=\{a^ub^v \mid 0 \leq u \leq p-1, 0\leq v \leq p^2-1\}$. 
Note that 
\[
h(c_{\left((i+jk)\pmod  p\right)p+jp^2+k})=a^{k}b^{ip+j}(c_{\left((i+jk)\pmod  p\right)p+jp^2+k})=a^{k+ip}b^{j}(c_{\left((i+jk)\pmod  p\right)p+jp^2+k}),
\]
as $a^{p} = b^{p}$.  We first show that $h(X)$ intersects $f$ simply. By Definitions \ref{D21} and \ref{D22}, it suffices to prove that for each  
$\beta_t$, $0 \le t \le p^{3}-1$, there exists a unique $x \in \mathbb{F}_{q}$ such that $\beta_t(x)=h(x)$. Further, since  $$\{\beta_0, \beta_1, \ldots, \beta_{q-1}\}=\{a^ub^v \mid 0 \leq u \leq p-1, 0\leq v \leq p^2-1\}=\{a^{k}b^{ip+j}\mid 0 \leq i,j,k \leq p-1\},$$ it is equivalent to show that for each triple $(i,j,k)$, $0 \leq i,j,k \leq p-1$, there exists a unique $x \in \mathbb{F}_{q}$ such that
\[
a^{k} b^{ip+j}(x) = h(x).
\] 
From the definition of $h$, the element
\[
x := c_{((i+jk)\bmod p)p + jp^{2} + k}
\]
satisfies $a^{k} b^{ip+j}(x) = h(x)$. Suppose, for the contradiction, there exists another element
\[
x' = c_{((i'+j'k')\bmod p)p + j'p^{2} + k'},
\qquad (i',j',k') \neq (i,j,k),~ 0 \leq i',j',k' \leq p-1
\]
such that $a^{k} b^{ip+j}(x') = h(x')$.  
By definition of $h$,
\[
h(x') = a^{k'} b^{i'p+j'}(x'),
\]
and therefore
\[
a^{k'} b^{i'p+j'}(x') = a^{k} b^{ip+j}(x').
\]
This is a contradiction as all permutations in $\{a^{k}b^{ip+j}\mid 0 \leq i,j,k \leq p-1\}=\{\beta_0,\beta_1,\ldots,\beta_{q-1}\}$ are distinct and $(\beta_0,\beta_1,\ldots,\beta_{q-1})$ is a permutation polynomial tuple. Thus, for each $\beta_t$, the equation $\beta_t(x)=h(x)$ has a unique solution in $\mathbb{F}_{q}$, and therefore $h$ intersects $f$ simply.

 Next, it remains to show that $h(X)$ is a permutation. Note that
 \begin{equation*}
 \begin{split}
 h(c_{\left((i+jk)\pmod  p\right)p+jp^2+k})&=a^{k+ip}b^{j}(c_{\left((i+jk)\pmod  p\right)p+jp^2+k})
 \\&=a^{k+ip}b_k^{(1+kp)j}(c_{\left((i+jk)\pmod  p\right)p+jp^2+k})
 \end{split}
 \end{equation*}
 as $b=\displaystyle \prod_{\ell=0}^{p-1}b_{\ell}^{1+\ell p}$ and Equation \eqref{e22} implies $b_{\ell}(c_t)=c_t$ whenever $t \not \equiv \ell \pmod  p$. We also have $a^p=b^p$, thus
 \[
  h(c_{\left((i+jk)\pmod  p\right)p+jp^2+k})=a^{k+p(i+kj)}b_k^{j}(c_{\left((i+jk)\pmod  p\right)p+jp^2+k}):=c_{t'}.
 \]
 

To determine the value of $t'$, we consider the following two cases.

\textbf{Case 1:} In this case, we assume that $ 0 \leq j \leq \frac{p-1}{2} $. If $j=0$, then 
\[
b_k^{j}(c_{\left((i+jk)\pmod  p\right)p+jp^2+k})=c_{\left((i+jk)\pmod  p\right)p+jp^2+k}=c_{\left((i+jk)\pmod  p\right)p+2jp^2+k}.
\]
We now assume that $ 1 \leq j \leq \frac{p-1}{2} $. For any $1 \leq w \leq j-1$, it can be seen that
\begin{equation*}
\begin{split}
\left((i+jk)\pmod  p\right)p+(j+w)p^2+k &\leq p^2-p+(2j-1)p^2+p-1\\&\leq p^2-p+(p-2)p^2+p-1< p^3-p^2.
\end{split}
\end{equation*}
Therefore, after using Equation \eqref{e22} $j$ times, we obtain 
 $$b_{k}^{j}(c_{\left((i+jk)\pmod  p\right)p+jp^2+k})=c_{\left((i+jk)\pmod  p\right)p+2jp^2+k}.$$
  Consequently, \begin{equation*}
  \begin{split}
h(c_{\left((i+jk)\pmod  p\right)p+jp^2+k})&=a^{k+p(i+kj)}b_k^{j}(c_{\left((i+jk)\pmod  p\right)p+jp^2+k})\\&=a^{k+p(i+kj)}(c_{\left((i+jk)\pmod  p\right)p+2jp^2+k}).
\end{split}
\end{equation*}
 Since $2jp^2 \leq \left((i+jk)\pmod  p\right)p+2jp^2+k \leq (p-1)p+2jp^2+p-1=(2j+1)p^2-1$, so Equation \eqref{tempe4} implies
\[
a^{k+p(i+kj)}(c_{\left((i+jk)\pmod  p\right)p+2jp^2+k})=c_{2jp^2+\left(\left((i+jk)\pmod  p\right)p+k+k+p(i+kj)\right) \pmod {p^2}} 
\]
and therefore,
\[
h(c_{\left((i+jk)\pmod  p\right)p+jp^2+k})=c_{2jp^2+\left(\left((i+jk)\pmod  p\right)p+k+k+p(i+kj)\right) \pmod {p^2}}. 
\]
Thus in this case $t'={2jp^2+\left(\left((i+jk)\pmod p\right)p+2k+p(i+kj)\right) \pmod {p^2}}$ and so
\[
  h(c_{\left((i+jk)\pmod  p\right)p+jp^2+k})=c_{2jp^2+\left(\left((i+jk)\pmod  p\right)p+2k+p(i+kj)\right) \pmod {p^2}}.
 \]

\textbf{Case 2:} Now, we assume that $ \frac{p+1}{2} \leq j \leq p-1 $. Note that
\[
b_{k}^{j}(c_{\left((i+jk)\pmod  p\right)p+jp^2+k})=b_{k}^{2j-p+1}(b_{k}^{p-1-j}(c_{\left((i+jk)\pmod  p\right)p+jp^2+k})).
\]
If $j=p-1$, then
\[
b_{k}^{p-1-j}(c_{\left((i+jk)\pmod  p\right)p+jp^2+k})=c_{\left((i+jk)\pmod  p\right)p+jp^2+k}=c_{\left((i+jk)\pmod  p\right)p+(p-1)p^2+k}.
\] 
Notice that for any $1 \leq w_1 \leq p-1-j$, we have
\[
\left((i+jk)\pmod  p\right)p+(j+w_1-1)p^2+k <p^3-p^2.
\] 
Thus, Equation \eqref{e22} is applied $(p-1-j)$ times to obtain
\[
b_{k}^{p-1-j}(c_{\left((i+jk)\pmod  p\right)p+jp^2+k})=c_{\left((i+jk)\pmod  p\right)p+(p-1)p^2+k}.
\]
 We further proceed this case by dividing it into two subcases, namely, $(i+jk)\pmod  p \leq p-2$ and $(i+jk)\pmod  p =p-1$. 

\textbf{Subcase 2.1:} We first assume that $(i+jk) \pmod  p \leq p-2$. In this subcase, we have 
\begin{equation*}
\begin{split}
p^3-p^2 &\leq \left((i+jk)\pmod  p\right)p+(p-1)p^2+k \\&\leq (p-2)p+p^3-p^2+k=p^3-p+k-p<p^3-p.
\end{split}
\end{equation*}
Thus using definition of $b_k$ from Equation \eqref{e22}, we get
\[
b_{k}(c_{\left((i+jk)\pmod  p\right)p+(p-1)p^2+k})=c_{\left((i+jk)\pmod  p\right)p+(p-1)p^2+k-p^3+p^2+p}=c_{\left((i+jk)\pmod  p\right)p+k+p}.
\]
Since for any $1 \leq w_2 \leq 2j-p$, 
\begin{equation*}
\begin{split}
\left((i+jk)\pmod  p\right)p+k+p+(w_2-1)p^2 &\leq \left((i+jk)\pmod  p\right)p+k+p+(2j-p-1)p^2 \\&\leq p^2-p+k+p+p^3-3p^2\\&=p^3-p^2+k-p^2<p^3-p^2,
\end{split}
\end{equation*}
thus, $b_{k}^{2j-p+1}(c_{\left((i+jk)\pmod  p\right)p+(p-1)p^2+k})=b_{k}^{2j-p}(c_{\left((i+jk)\pmod  p\right)p+k+p})=c_{\left((i+jk)\pmod  p\right)p+k+p+(2j-p)p^2}$. Combining all together,  
\begin{equation*}
\begin{split}
b_{k}^{j}(c_{\left((i+jk)\pmod  p\right)p+jp^2+k})&=b_{k}^{2j-p+1}(b_{k}^{p-1-j}(c_{\left((i+jk)\pmod  p\right)p+jp^2+k}))\\&=b_{k}^{2j-p+1}(c_{\left((i+jk)\pmod  p\right)p+(p-1)p^2+k})\\&=b_{k}^{2j-p}(b_k(c_{\left((i+jk)\pmod  p\right)p+(p-1)p^2+k}))\\&=b_{k}^{2j-p}(c_{\left((i+jk)\pmod  p\right)p+k+p})
\\&=c_{\left((i+jk)\pmod  p\right)p+k+p+(2j-p)p^2}.
\end{split}
\end{equation*}
Since $(2j-p)p^2 < \left((i+jk)\pmod  p\right)p+k+p+(2j-p)p^2 \leq (p-2)p+p-1+p+(2j-p)p^2=(2j-p+1)p^2-1$, therefore, applying Equation \eqref{tempe4},
\begin{equation*}
\begin{split}
h(c_{\left((i+jk)\pmod  p\right)p+jp^2+k})&=a^{k+p(i+kj)}b_k^{j}(c_{\left((i+jk)\pmod  p\right)p+jp^2+k})\\&=c_{(2j-p)p^2+\left(\left((i+jk)\pmod  p\right)p+k+p+k+p(i+kj)\right) \pmod {p^2}}.
\end{split}
\end{equation*}

\textbf{Subcase 2.2:} Let $(i+jk)\pmod  p =p-1$. In this subcase, we have $$ \left((i+jk)\pmod  p\right)p+(p-1)p^2+k \geq (p-1)p+p^3-p^2+k=p^3-p+k \geq p^3-p.$$ Therefore, 
\[
b_k(c_{\left((i+jk)\pmod  p\right)p+(p-1)p^2+k})=c_{\left((i+jk)\pmod  p\right)p+(p-1)p^2+k-p^3+p}=c_{\left((i+jk)\pmod  p\right)p-p^2+k+p}.
\]
Furthermore, we have $\left((i+jk)\pmod  p\right)p-p^2+k+p+(w_3-1)p^2<p^3-p^2$ for any positive integer $ 1 \leq w_3 \leq 2j-p$. Thus, using the similar argument as in Subcase 2.1, we get 
\begin{equation*}
\begin{split}
b_k^{2j-p+1}(c_{\left((i+jk)\pmod  p\right)p+(p-1)p^2+k})&=b_k^{2j-p}(c_{\left((i+jk)\pmod  p\right)p-p^2+k+p})
\\&=c_{\left((i+jk)\pmod  p\right)p-p^2+k+p+(2j-p)p^2}
\end{split}\end{equation*}
and 
\[
b_{k}^{j}(c_{\left((i+jk)\pmod  p\right)p+jp^2+k})=c_{\left((i+jk)\pmod  p\right)p-p^2+k+p+(2j-p)p^2}.
\]
Consequently, from Equation \eqref{tempe4},
\begin{equation*}
\begin{split}
a^{k+p(i+kj)}b_k^{j}(c_{\left((i+jk)\pmod  p\right)p+jp^2+k})&=c_{(2j-p)p^2+\left(\left((i+jk)\pmod  p\right)p-p^2+k+p+k+p(i+kj)\right)  \pmod {p^2}}
\\&=c_{(2j-p)p^2+\left(\left((i+jk)\pmod  p\right)p+k+p+k+p(i+kj)\right) \pmod {p^2}}
\end{split}
\end{equation*}
as $(2j-p)p^2 \leq \left((i+jk)\pmod  p\right)p-p^2+k+p+(2j-p)p^2=(p-1)p-p^2+k+p+(2j-p)p^2=k+(2j-p)p^2<(2j-p+1)p^2-1$.

So far, we have
\[
h(c_{\left((i+jk)\pmod  p\right)p+jp^2+k})=\begin{cases}

       c_{2jp^2+\left(\left((i+jk)\pmod  p\right)p+2k+p(i+kj)\right)  \pmod {p^2}} & \text{ if }  j \leq \frac{p-1}{2}, \\
       c_{(2j-p)p^2+\left(\left((i+jk)\pmod p\right)p+2k+p+p(i+kj)\right) \pmod {p^2}} & \text{ if } j \geq \frac{p+1}{2}.
       	\end{cases}
\]
Now, suppose that there exist $(i_1, j_1, k_1) \neq (i_2, j_2, k_2)$ with $0 \leq j_1, j_2 \leq \frac{p-1}{2}$ such that 
$$h(c_{\left((i_1+j_1k_1)\pmod  p\right)p+j_1p^2+k_1})=h(c_{\left((i_2+j_2k_2)\pmod  p\right)p+j_2p^2+k_2}).$$ This gives
\[
c_{2j_1p^2+\left(\left((i_1+j_1k_1)\pmod  p\right)p+2k_1+p(i_1+k_1j_1)\right) \pmod {p^2}}=c_{2j_2p^2+\left(\left((i_2+j_2k_2)\pmod  p\right)p+2k_2+p(i_2+k_2j_2)\right) \pmod {p^2}}, 
\]
or,
\begin{equation*}
\begin{split}
2j_1p^2+\left(\left((i_1+j_1k_1)\pmod  p\right)p+2k_1+p(i_1+k_1j_1)\right) \pmod {p^2}=&2j_2p^2+(((i_2+j_2k_2)\pmod p)p\\&+2k_2+p(i_2+k_2j_2)) \pmod {p^2}.
\end{split}
\end{equation*}
By taking modulo $p$ both sides, we get $k_1=k_2$ as $p$ is an odd prime. Now, if possible $j_1 \neq j_2$. Without loss of generality we may assume $j_1 < j_2$. Then clearly
\begin{equation*}
\begin{split}
2j_1p^2+\left(\left((i_1+j_1k_1)\pmod  p\right)p+2k_1+p(i_1+k_1j_1)\right) \pmod {p^2}<&2j_2p^2+(((i_2+j_2k_2)\pmod p)p\\&+2k_2+p(i_2+k_2j_2)) \pmod {p^2},
\end{split}
\end{equation*}
as $0\leq \left(\left((i_1+j_1k_1)\pmod  p\right)p+2k_1+p(i_1+k_1j_1)\right) \pmod {p^2}, (((i_2+j_2k_2)\pmod p)p+2k_2+p(i_2+k_2j_2)) \pmod {p^2} \leq p^2-1$. This implies 
\[h(c_{\left((i_1+j_1k_1)\pmod p\right)p+j_1p^2+k_1})\neq h(c_{\left((i_2+j_2k_2)\pmod p\right)p+j_2p^2+k_2}).\]
Thus $j_1=j_2$.  Since $k_1=k_2$, $j_1=j_2$, 
\[
i_1+j_1k_1+i_1\equiv i_2+j_2k_2+i_2 \pmod  p,
\]
which gives that $2i_1 \equiv 2i_2  \pmod  p$, thus $i_1=i_2$ as $p$ is an odd prime. In a similar way if $(i_1, j_1, k_1) \neq (i_2, j_2, k_2)$ with $j_1, j_2 \geq \frac{p+1}{2}$, we can show that  $$h(c_{\left((i_1+j_1k_1)\pmod  p\right)p+j_1p^2+k_1})\neq h(c_{\left((i_2+j_2k_2)\pmod  p\right)p+j_2p^2+k_2}).$$ Next, suppose, for the sake of contradiction, that there exist  $(i_1, j_1, k_1) \neq (i_2, j_2, k_2)$ with $0 \leq j_1 \leq \frac{p-1}{2}$ and $j_2 \geq \frac{p+1}{2}$ (without loss of generality) such that  $$h(c_{\left((i_1+j_1k_1)\pmod  p\right)p+j_1p^2+k_1})= h(c_{\left((i_2+j_2k_2)\pmod  p\right)p+j_2p^2+k_2}).$$  Again using the similar argument as above if $2j_1 \neq 2j_2-p$, then 
\[h(c_{\left((i_1+j_1k_1)\pmod  p\right)p+j_1p^2+k_1})\neq h(c_{\left((i_2+j_2k_2)\pmod  p\right)p+j_2p^2+k_2}).\]
Therefore, equality holds only if $2j_1=2j_2-p$. This gives $2j_1 \equiv 2j_2 \pmod  p$ or $j_1=j_2$ as $p$ is an odd prime, which is a contradiction. Hence $h$ is a permutation, which completes the proof.
\end{proof}

\section{Enumeration of Permutation group polynomials}\label{S4}
 We call two permutation group polynomials over the finite field $\F_q$ of the same form if their associated groups, as subgroups of $\mathfrak{S}_q$, are isomorphic. Counting permutation group polynomials of a given form is an interesting and intriguing problem of group theory and combinatorics. In this section, we enumerate permutation group polynomials of the form proposed in Section \ref{S3}, as well as, those permutation group polynomials  that are of the form of three families of permutation group polynomials constructed in \cite{HK}. Moreover, we give the exact number of permutation group polynomials that are equivalent to these permutation group polynomials. 
 
 Let $G$ be a group defined in one of  Theorem \ref{T31}, Lemma \ref{HK1}, Lemma \ref{HK2} or Lemma \ref{HK3} and $f(X_1,X_2)$ be any permutation group polynomial corresponding to $G$. Then clearly we have that $G=\langle a,b \rangle$, where $a$ is the product of $p^{n-r_1}$ disjoint cycles each of length $p^{r_1}$, $b$ is the product of $p^{n-r_2}$ disjoint cycles each of length $p^{r_2}$, $ab=b^{r}a$ and $a^{p^{n-r_2}}=b^{p^{n-r_1}}$ for some integers $r,r_1,r_2$ with $1 \le r_1 \le r_2 \le n-1$ and $\gcd(r,|b|)=1$. Moreover, $|G|=p^n$ and every non identity permutation of $G$ has no fixed point in $\F_q$. Let $g(X_1,X_2) \in \F_q[X_1,X_2]$ be a permutation group polynomial of the form as $f(X_1,X_2)$. Thus, subgroup $H < \mathfrak{S}_q$ corresponding to $g(X_1,X_2)$ will be isomorphic to $G$.  Therefore, we have $H=\langle a_1,b_1 \rangle$, $|H|=p^n$ and every non identity permutation of $H$ has no fixed point in $\F_q$, where  $a_1$ is the product of $p^{n-r_1}$ disjoint cycles each of length $p^{r_1}$,  $b_1$ is the product of $p^{n-r_2}$ disjoint cycles each of length $p^{r_2}$, $a_1b_1=b_1^{r}a_1$ and  $a_1^{p^{n-r_2}}=b_1^{p^{n-r_1}}$. Since $b$ and $b_1$ have the same cyclic structure, therefore, there exists $\sigma \in \mathfrak{S}_q$ such that $\sigma b_1 \sigma^{-1} =b$ and $\sigma a_1 \sigma^{-1} =a'$ for some $a' \in \mathfrak{S}_q$.  Clearly $a'$ has the same cyclic structure as $a$. Furthermore,  we have $a'b(a')^{-1}=b^{r}=aba^{-1}$, $a^{p^{n-r_2}}=(a')^{p^{n-r_2}}=b^{p^{n-r_1}}$, $|\langle a',b \rangle|=p^n$ and every non identity permutation of $\langle a',b \rangle$ has no fixed point. Based on this discussion, we make the following remark. 
  \begin{rmk}\label{R'_conjugate}
 $G$ and $H$ are conjugate to each other in $\mathfrak{S}_q$ if and only if  $\langle a,b \rangle$ and $\langle a',b \rangle$  are conjugate in $\mathfrak{S}_q$.
 \end{rmk}
  The following lemma shows that $\langle a,b \rangle$ and  $\langle a',b \rangle$ are conjugate in $\mathfrak{S}_q$.
\begin{lem}\label{G_conjugate}
 Let $q=p^n$ and $a,a'\in \mathfrak{S}_q$ be two permutations that are the product of $p^{n-r_1}$ disjoint cycles each of length $p^{r_1}$, where $1 \le r_1 \le n-1$ is a non negative integer. For $1 \le r_1 \le r_2 \le n-1$, consider $b \in \mathfrak{S}_q$
 \[
 b=\displaystyle  \prod_{i=0}^{p^{n-r_2}-1}(c_{ip^{r_2}+0},c_{ip^{r_2}+1},\ldots,c_{ip^{r_2}+p^{r_2}-1})
 \] such that $aba^{-1}=b^r=a'b(a')^{-1}$ and $a^{p^{n-r_2}}=b^{p^{n-r_1}}=(a')^{p^{n-r_2}}$, where $\gcd(r,|b|)=1$. Moreover, let the subgroups $\langle a,b \rangle$ and $\langle a',b \rangle$ have cardinality $q$ and every non identity permutation of these subgroups has no fixed point. Then $\langle a,b \rangle$ and $\langle a',b \rangle$ are conjugate in $\mathfrak{S}_q$.
 \end{lem}
 \begin{proof}
 Using the given properties of $a$ and $a'$, we first find $a$ and $a'$, explicitly. Since $ab=b^{r}a$, we have 
 \begin{equation}\label{eH1}
 \begin{split}
 \displaystyle  \prod_{i=0}^{p^{n-r_2}-1}\left(a(c_{ip^{r_2}+0}),a(c_{ip^{r_2}+1}),\ldots,a(c_{ip^{r_2}+p^{r_2}-1})\right)=b^r=\displaystyle  \prod_{i=0}^{p^{n-r_2}-1}&\left(c_{ip^{r_2}+0},c_{ip^{r_2}+r\pmod{p^{r_2}}},\ldots,\right.\\&\left. c_{ip^{r_2}+(r(p^{r_2}-1))\pmod{p^{r_2}}}\right).
 \end{split}
 \end{equation}
 First, note that for any $1 \le w \le p^{n-r_2}-1$, $a^w \not \in \langle b \rangle$, otherwise $|\langle a,b \rangle| \leq p^n-p^{r_2}  <p^n$.  We now assume that $a(c_0)=c_{jr \pmod{p^{r_2}}}$ for some $0 \le j \le p^{r_2}-1$. Then $(b^r)^{-j}a(c_0)=c_0$, which is a contradiction as $(b^r)^{-j}a \neq I$. Thus, $a(c_0)=c_{i_1p^{r_2}+j_1r \pmod{p^{r_2}}}$, where $1 \le i_1 \le p^{n-r_2}-1$ and $0 \le j_1 \le p^{r_2}-1$. This together with Equation \eqref{eH1} imply 
 \begin{equation}\label{eH2}
 a(c_k)=c_{i_1p^{r_2}+(j_1+k)r\pmod{p^{r_2}}}
 \end{equation}
  for any $0 \leq k \leq p^{r_2}-1$. Now if $a(c_{i_1p^{r_2}+j_1r \pmod{p^{r_2}}})=c_{jr \pmod{p^{r_2}}}$ for some $0 \le j \le p^{r_2}-1$. Then $(b^r)^{-j}a^2(c_0)=c_0$, which is again a contradiction. If  $a(c_{i_1p^{r_2}+j_1r \pmod{p^{r_2}}})=c_{i_1p^{r_2}+jr \pmod{p^{r_2}}}$, then $(b^r)^{j_1-j}a(c_{i_1p^{r_2}+j_1r \pmod{p^{r_2}}})=c_{i_1p^{r_2}+j_1r \pmod{p^{r_2}}}$, which is not possible. Therefore, we have $a(c_{i_1p^{r_2}+j_1r \pmod{p^{r_2}}})=c_{i_2p^{r_2}+j_2r \pmod{p^{r_2}}}$, where $0 \le i_2 \le p^{n-r_2}-1$, $i_2 \not \in \{0,i_1\}$ and $0 \le j_2 \le p^{r_2}-1$. Further, this gives $$a\left(c_{i_1p^{r_2}+(j_1+k)r \pmod{p^{r_2}}}\right)=c_{i_2p^{r_2}+(j_2+k)r \pmod{p^{r_2}}}$$ for $ 0 \le k \le p^{r_2}-1$. Continuing in this way, for each $1 \le u \le p^{n-r_2}-2$, and $0 \le k \le p^{r_2}-1$, we obtain 
 \begin{equation}\label{eH3}
 a\left(c_{i_up^{r_2}+(j_u+k)r \pmod{p^{r_2}}}\right)=c_{i_{u+1}p^{r_2}+(j_{u+1}+k)r \pmod{p^{r_2}}},
 \end{equation}
 where $1 \leq i_1,i_2,\ldots,i_{p^{n-r_2}-1} \leq p^{n-r_2}-1$, $0 \le j_1,\ldots,j_{p^{n-r_2}-1} \le p^{r_2}-1$ such that $\{i_1,i_2,\ldots,i_{p^{n-r_2}-1}\}=\{1,2,\ldots,p^{n-r_2}-1\}$. If $p^{n-r_1}=p^{r_2}$, then we stop the process otherwise we proceed further. We have $a^{p^{n-r_2}}=b^{p^{n-r_1}}$, which in particular gives $a^{p^{n-r_2}}(c_{v})=b^{p^{n-r_1}}(c_v)$ for any $0 \le v \le p^{n-r_1}-1$. Further the equality $a^{p^{n-r_2}}(c_{v})=b^{p^{n-r_1}}(c_v)$ yields $a(a^{p^{n-r_2}-1}(c_{v}))=c_{v+p^{n-r_1}}$, that is, 
 $$a\left(c_{i_{p^{n-r_2}-1}p^{r_2}+(j_{p^{n-r_2}-1}+v)r \pmod{p^{r_2}}}\right)=c_{v+p^{n-r_1}}$$ as $v+p^{n-r_1} \le p^{r_2}-1$. Moreover, from Equations \eqref{eH2} and \eqref{eH3}, we have  
  \[
 a(c_{v+p^{n-r_1}})=c_{i_{1}p^{r_2}+(j_{1}+v+p^{n-r_1})r \pmod{p^{r_2}}},
 \]
 and 
 \[
 a\left(c_{i_up^{r_2}+(j_u+v+p^{n-r_1})r \pmod{p^{r_2}}}\right)=c_{i_{u+1}p^{r_2}+(j_{u+1}+v+p^{n-r_1})r \pmod{p^{r_2}}},
 \]
for all $1 \le u \le p^{n-r_2}-2$ and $0 \le v \le p^{n-r_1}-1$. Again due to $a^{p^{n-r_2}}=b^{p^{n-r_1}}$, we get  $$a\left(c_{i_{p^{n-r_2}-1}p^{r_2}+(j_{p^{n-r_2}-1}+v+p^{n-r_1})r \pmod{p^{r_2}}}\right)=c_{v+2p^{n-r_1}}.$$ Similar argument as above gives 
\[
 a(c_{v+2p^{n-r_1}})=c_{i_{1}p^{r_2}+(j_{1}+v+2p^{n-r_1})r \pmod{p^{r_2}}},
 \]
and 
 \[
 a\left(c_{i_up^{r_2}+(j_u+v+2p^{n-r_1})r \pmod{p^{r_2}}}\right)=c_{i_{u+1}p^{r_2}+(j_{u+1}+v+2p^{n-r_1})r \pmod{p^{r_2}}}
 \]
 for all $1 \le u \le p^{n-r_2}-2$ and $0 \le v \le p^{n-r_1}-1$. Finally, following this process for any $1 \le u \le p^{n-r_2}-2$, $0 \le v \le p^{n-r_1}-1$ and $0 \le \ell \le p^{r_1+r_2-n}-1$, we have 
 \[
 a(c_{v+\ell p^{n-r_1}})=c_{i_{1}p^{r_2}+(j_{1}+v+\ell p^{n-r_1})r \pmod{p^{r_2}}},
 \]
  \[
 a\left(c_{i_up^{r_2}+(j_u+v+\ell p^{n-r_1})r \pmod{p^{r_2}}}\right)=c_{i_{u+1}p^{r_2}+(j_{u+1}+v+\ell p^{n-r_1})r \pmod{p^{r_2}}}
 \]
 and 
  $$a\left(c_{i_{p^{n-r_2}-1}p^{r_2}+(j_{p^{n-r_2}-1}+v+\ell p^{n-r_1})r \pmod{p^{r_2}}}\right)=c_{v+(\ell +1)\pmod {p^{r_1+r_2-n}}p^{n-r_1}}.$$ Thus, $a$ can be written as follows 
  \begin{equation*}
  \begin{split}
  a=\displaystyle \prod_{v=0}^{p^{n-r_1}-1}&\left(c_{v},c_{i_1p^{r_2}+(v+j_1)\pmod{p^{r_2}}},\ldots,c_{i_{p^{n-r_2}-1}p^{r_2}+(v+j_{p^{n-r_2}-1})\pmod{p^{r_2}}},c_{v+p^{n-r_1}},\right.\\&c_{i_1p^{r_2}+(v+j_1+p^{n-r_1})\pmod{p^{r_2}}},\ldots,c_{i_{p^{n-r_2}-1}p^{r_2}+(v+j_{p^{n-r_2}-1}+p^{n-r_1})\pmod{p^{r_2}}},\ldots,\\&c_{v+(p^{r_1+r_2-n}-1)p^{n-r_1})\pmod{p^{r_2}}},c_{i_1p^{r_2}+(v+j_1+(p^{r_1+r_2-n}-1)p^{n-r_1})\pmod{p^{r_2}}},\ldots,\\&\left. c_{i_{p^{n-r_2}-1}p^{r_2}+(v+j_{p^{n-r_2}-1}+(p^{r_1+r_2-n}-1)p^{n-r_1})\pmod{p^{r_2}}}\right).
  \end{split}
  \end{equation*}
  Similarly, one can get 
    \begin{equation*}
  \begin{split}
  a'=\displaystyle \prod_{v=0}^{p^{n-r_1}-1}&\left(c_{v},c_{i'_1p^{r_2}+(v+j'_1)\pmod{p^{r_2}}},\ldots,c_{i'_{p^{n-r_2}-1}p^{r_2}+(v+j'_{p^{n-r_2}-1})\pmod{p^{r_2}}},c_{v+p^{n-r_1}},\right. \\&c_{i'_1p^{r_2}+(v+j'_1+p^{n-r_1})\pmod{p^{r_2}}},\ldots,c_{i'_{p^{n-r_2}-1}p^{r_2}+(v+j'_{p^{n-r_2}-1}+p^{n-r_1})\pmod{p^{r_2}}},\ldots,\\&c_{v+(p^{r_1+r_2-n}-1)p^{n-r_1})\pmod{p^{r_2}}},c_{i'_1p^{r_2}+(v+j'_1+(p^{r_1+r_2-n}-1)p^{n-r_1})\pmod{p^{r_2}}},\ldots,\\&\left. c_{i'_{p^{n-r_2}-1}p^{r_2}+(v+j'_{p^{n-r_2}-1}+(p^{r_1+r_2-n}-1)p^{n-r_1})\pmod{p^{r_2}}}\right).
  \end{split}
  \end{equation*}
  Define a map $\sigma: \mathfrak{S}_q \rightarrow \mathfrak{S}_q$ as follows 
  $$\sigma(c_{v+\ell p^{n-r_1}})=c_{v+\ell p^{n-r_1}}$$ and 
  $$\sigma\left(c_{i_{u}p^{r_2}+(v+j_{u}+\ell p^{n-r_1})r\pmod{p^{r_2}}}\right)=c_{i'_{u}p^{r_2}+(v+j'_{u}+\ell p^{n-r_1})r\pmod{p^{r_2}}}$$
  for $1 \le u \le p^{n-r_2}-1$, $0 \le v \le p^{n-r_1}-1$ and $0 \le \ell  \le p^{r_1+r_2-n}-1$.  By definition of $\sigma$, it is clear that $\sigma a \sigma^{-1}=a'$. Now we will show that $\sigma b \sigma^{-1}=b$. Notice that 
  \begin{equation*}
  \begin{split}
  \sigma b \sigma^{-1}(c_{v+\ell p^{n-r_1}})= \sigma b(c_{v+\ell p^{n-r_1}})=\sigma (c_{(v+1+\ell p^{n-r_1})\pmod{p^{r_2}}})=c_{(v+1+\ell p^{n-r_1})\pmod{p^{r_2}}}=b(c_{v+\ell p^{n-r_1}}).
  \end{split}
  \end{equation*}
  Next we will verify for the elements of the form $c_{i'_{u}p^{r_2}+(v+j'_{u}+\ell p^{n-r_1})r\pmod{p^{r_2}}}$ for $1 \le u \le p^{n-r_2}-1$. We have 
    \begin{equation*}
  \begin{split}
  \sigma b \sigma^{-1}\left(c_{i'_{u}p^{r_2}+(v+j'_{u}+\ell p^{n-r_1})r\pmod{p^{r_2}}}\right)&= \sigma b\left(c_{i_{u}p^{r_2}+(v+j_{u}+\ell p^{n-r_1})r\pmod{p^{r_2}}}\right)\\&=\sigma \left(c_{i_{u}p^{r_2}+((v+j_{u}+\ell p^{n-r_1})r+1)\pmod{p^{r_2}}}\right)\\&=\sigma \left(c_{i_{u}p^{r_2}+(v+j_{u}+\ell p^{n-r_1}+r^{-1})r\pmod{p^{r_2}}}\right)\\&=c_{i'_{u}p^{r_2}+(v+j'_{u}+\ell p^{n-r_1}+r^{-1})r\pmod{p^{r_2}}}\\&=b\left(c_{i'_{u}p^{r_2}+(v+j'_{u}+\ell p^{n-r_1})r\pmod{p^{r_2}}}\right).
  \end{split}
  \end{equation*}
  Consequently, $\sigma \langle a,b \rangle \sigma^{-1}=\langle a',b \rangle$, which completes the proof.
\end{proof}
The following remark will play a crucial role in enumerating permutation group polynomials of the forms constructed in Section \ref{S3} and in \cite{HK}
\begin{rmk}\label{R_conjugate}
Let $g(X_1,X_2) \in \F_q[X_1,X_2]$ be a permutation group polynomial of the form as $f(X_1,X_2)$, where $f(X_1,X_2)$ is a permutation group polynomial corresponding to the group $G$ defined in one of  Theorem \ref{T31}, Lemma \ref{HK1}, Lemma \ref{HK2} or Lemma \ref{HK3}. Moreover, let $H <\mathfrak{S}_q$ be the corresponding group of $g(X_1,X_2)$. Then Remark \ref{R'_conjugate} and Lemma \ref{G_conjugate} yield that $H$ and $G$ are conjugate in $\mathfrak{S}_q$.   
\end{rmk}
We now introduce a notation that will be used in the following results. Let $G=\langle a,b\rangle $ be the subgroup defined in Theorem \ref{T31}. For fixed $i_1,i_2 \in  \{0,1,\ldots,p-1\}$, and $j_1,j_2 \in \{0,1,\ldots,p^{2}-1\}$, define 
\[
\mathcal{N}(i_1,j_1,i_2,j_2):=\{h \in \mathfrak{S}_q \mid hah^{-1}=b^{j_1}a^{i_1}, hbh^{-1}=b^{j_2}a^{i_2}\}.
\]

The following two lemmas will be used in the proof of the main theorem concerning the enumeration of permutation group polynomials of the form those constructed in Theorem~\ref{T31}.
\begin{lem}\label{Atmost_lemma_T31}
Let $q=p^3$, where $p$ is an odd prime, and $G=\langle a,b \rangle $ be the subgroup of $\mathfrak{S}_q$ defined as in Theorem \ref{T31}. Moreover, let $i_1,i_2 \in  \{0,1,\ldots,p-1\}$ and $j_1,j_2 \in \{0,1,\ldots,p^{2}-1\}$ be fixed integers. Then $|\mathcal{N}(i_1,j_1,i_2,j_2)| \leq  p^3$.
\end{lem}
\begin{proof} If $\mathcal{N}(i_1,j_1,i_2,j_2)=\emptyset$, then we are done. Let us assume that $\mathcal{N}(i_1,j_1,i_2,j_2) \neq \emptyset$ and define a map $\psi: \mathcal{N}(i_1,j_1,i_2,j_2) \rightarrow \F_q$ as follows $\psi(h)=h(c_0)$. To show that $|\mathcal{N}(i_1,j_1,i_2,j_2)| \leq  p^3$, it suffices to prove that $\psi$ is an injective map. Suppose that $\psi(h_1)=\psi(h_2)$ for some $h_1,h_2 \in \mathcal{N}(i_1,j_1,i_2,j_2)$. This gives us $h_1(c_0)=h_2(c_0)$ and $$h_1ah_1^{-1}=b^{j_1}a^{i_1},~~ h_1bh_1^{-1}=b^{j_2}a^{i_2}, ~~h_2ah_2^{-1}=b^{j_1}a^{i_1},~~ h_2bh_2^{-1}=b^{j_2}a^{i_2}.$$ Thus, 
\begin{equation}\label{e41}
h_1ah_1^{-1}=h_2ah_2^{-1},
\end{equation}
\begin{equation}\label{e42}
 h_1bh_1^{-1}=h_2bh_2^{-1}
 \end{equation}
  and $h_1(c_0)=h_2(c_0)$. We will show that $h_1=h_2$. Let $c_t \in \F_q$ be any element, where $0 \leq t \leq p^3-1$. It is clear that there exist $1 \leq r_t \leq p$ and $0 \leq s_t \leq p^2-1$ depending on $t$ such that $t=(r_t-1)p^2+s_t$. We first prove that $h_1(c_{(r_t-1)p^2})=h_2(c_{(r_t-1)p^2})$ by applying induction on $r_t$ and using Equation \eqref{e42}. If $r_t=1$, then there is nothing to prove. Now assume that $h_1(c_{(\ell-1)p^2})=h_2(c_{(\ell-1)p^2})$ for all positive integers $\ell <r_t$. Equation \eqref{e42} implies $ h_1bh_1^{-1}(h_1(c_{(\ell-1)p^2}))=h_2bh_2^{-1}(h_1(c_{(\ell-1)p^2}))$, that is, $ h_1bh_1^{-1}(h_1(c_{(\ell-1)p^2}))=h_2bh_2^{-1}(h_2(c_{(\ell-1)p^2}))$. This gives that $ h_1b(c_{(\ell-1)p^2})=h_2b(c_{(\ell-1)p^2})$. Since $\ell <r_t\leq p$, we have $(\ell-1)p^2<p^3-p^2$, thus using Equation \eqref{e22}, $b(c_{(\ell-1)p^2})=c_{\ell p^2}$ implies $h_1(c_{\ell p^2})=h_2(c_{\ell p^2})$. Therefore, by induction, $h_1(c_{(r_t-1)p^2})=h_2(c_{(r_t-1)p^2})$. Now we use Equation \eqref{e41} and induction on $s_t$ to prove $h_1(c_{(r_t-1)p^2+s_t})=h_2(c_{(r_t-1)p^2+s_t})$. If $s_t=0$, then we are done. Now assume that the result is true for any $\ell'<s_t$, that is, $h_1(c_{(r_t-1)p^2+\ell'})=h_2(c_{(r_t-1)p^2+\ell'})$. Employing Equation \eqref{e41}, we get $h_1a(c_{(r_t-1)p^2+\ell'})=h_2a(c_{(r_t-1)p^2+\ell'})$. Since $(r_t-1)p^2+\ell'<(r_t-1)p^2+p^2-1=r_tp^2-1$, we obtain $h_1(c_{(r_t-1)p^2+\ell'+1})=h_2(c_{(r_t-1)p^2+\ell'+1})$ by using Equation \eqref{e21}. Thus, by induction we have $h_1(c_{(r_t-1)p^2+s_t})=h_2(c_{(r_t-1)p^2+s_t})$. Consequently $h_1=h_2$ and hence $\psi$ is injective, which completes the proof.

\end{proof}

\begin{lem}\label{Atleast_lemma_T31}
Let $q=p^3$, where $p$ is an odd prime, and $G=\langle a,b \rangle$ be a subgroup of $\mathfrak{S}_q$ defined as in Theorem \ref{T31}. Furthermore,  let $i_1,i_2 \in \{0,1,\ldots,p-1\}$ and $j_1,j_2 \in \{0,1,\ldots,p^{2}-1\}$. Then $\mathcal{N}(i_1,j_1,i_2,j_2) \neq \emptyset$ if and only if 
\begin{equation*}
\begin{split}
& j_1+i_1 \not \equiv  0  \pmod  p,\\& 
  j_2+i_2 \not \equiv  0 \pmod  p,
\end{split}
\end{equation*}
 $i_2\equiv i_1-1 \pmod  p$ and $j_2 \equiv j_1+1 \pmod  p$. In this case,  $|\mathcal{N}(i_1,j_1,i_2,j_2)|=p^3$.
\end{lem}
\begin{proof}
We first assume that $\mathcal{N}(i_1,j_1,i_2,j_2) \neq \emptyset$, where $i_1,i_2 \in \{0,1,\ldots,p-1\}$ and $j_1,j_2 \in \{0,1,\ldots,p^{2}-1\}$. We will show that $ j_1+i_1 \not \equiv  0  \pmod  p$, $j_2+i_2 \not \equiv  0 \pmod  p$, $i_2\equiv i_1-1 \pmod  p$ and $j_2 \equiv j_1+1 \pmod  p$. Let $h \in \mathcal{N}(i_1,j_1,i_2,j_2)$. Thus, we have 
\begin{equation}\label{s41}
\begin{split}
&hah^{-1}=b^{j_1}a^{i_1}\\
&hbh^{-1}=b^{j_2}a^{i_2},
\end{split}
\end{equation}
which gives $\mid b^{j_1}a^{i_1}\mid =\mid b^{j_2}a^{i_2}\mid =p^2$ as $|a|=|b|=p^2$. Therefore, we have to exclude all the possibilities of $i_1,j_1,i_2,j_2$ such that $(b^{j_1}a^{i_1})^p=I$ or $(b^{j_2}a^{i_2})^p=I$. We make use of the relation $ab=b^{p^2-p+1}a$ and do some simplifications to obtain 
\[
(b^{j_1}a^{i_1})^p=b^{j_1N+i_1p},
\]
where 
$$
N=\begin{cases} \dfrac{(p^2-p+1)^{i_1p}-1}{(p^2-p+1)^{i_1}-1} & \text{ if } i_1 \neq 0,\\p 
 & \text{ if } i_1=0.\end{cases}$$
This implies that  $(b^{j_1}a^{i_1})^p=I$ if and only if $j_1N+i_1p \equiv 0 \pmod {p^2}$.
 It is trivial that $N \equiv p \pmod {p^2}$ if $i_1=0$. Now we shall show that $N \equiv p \pmod {p^2}$ also holds for $1 \leq i_1 \leq p-1$.  Notice that for $i_1 \neq 0$ 
\[
N=1+(p^2-p+1)^{i_1}+(p^2-p+1)^{2i_1}+\cdots+(p^2-p+1)^{(p-1)i_1}.
\]
We can write $p^2-p+1=wp+1$, where $w=p-1$. For any positive integer $k$, $$(1+wp)^k=\displaystyle \sum_{t=0}^{k}  \binom{k}{t} (wp)^t \equiv 1+kwp \pmod {p^2}.$$Thus, 
\[
N \equiv \displaystyle p+i_1pw\sum_{s=1}^{p-1}s \equiv p+i_1pw \frac{p(p-1)}{2} \equiv p \pmod {p^2}.
\]
 Hence,  $(b^{j_1}a^{i_1})^p=I$ if and only if $j_1N+i_1p \equiv 0 \pmod {p^2}$ if and only if $j_1+i_1 \equiv 0 \pmod p$. Consequently $|b^{j_1}a^{i_1}|=p^2$ implies $j_1+i_1 \not \equiv 0 \pmod  p$. Similarly, for $|b^{j_2}a^{i_2}|=p^2$, we can prove that $j_2+i_2 \not \equiv 0 \pmod  p$.

It remains to show that $i_2\equiv i_1-1 \pmod  p$ and $j_2 \equiv j_1+1 \pmod  p$. After raising $p$-th power to both sides of the equations in System \eqref{s41} and using $a^p=b^p$, we get $(b^{j_1}a^{i_1})^p=(b^{j_2}a^{i_2})^p$. This gives us  
\[
b^{j_1N+i_1p}=b^{j_2N'+i_2p},
\]
where $$
N'=\begin{cases} \dfrac{(p^2-p+1)^{i_2p}-1}{(p^2-p+1)^{i_2}-1} & \text{ if } i_2 \neq 0,\\p 
 & \text{ if } i_2=0.\end{cases}$$
  However, we know that equality $b^{j_1N+i_1p}=b^{j_2N'+i_2p}$ holds if and only if  $j_1N+i_1p \equiv j_2N'+i_2p \pmod {p^2}$.
 From the above discussion and definitions of $N$ and $N'$, it is clear that $N \equiv p \equiv N' \pmod {p^2}$ for $0 \leq i_1,i_2 \leq p-1$. Thus,  $j_1N+i_1p \equiv j_2N'+i_2p \pmod {p^2}$ if and only if $j_1+i_1 \equiv j_2+i_2 \pmod p$ or equivalently, $i_1-i_2 \equiv j_2-j_1 \pmod p$. Let $i_1-i_2 \equiv j_2-j_1 \equiv r \pmod p$ for some $0 \leq r \leq p-1$. It allows us to write $i_1=i_2+r+pt_1$ and $j_1=j_2-r+pt_2$ for some integers $t_1$ and $t_2$.
 
 Furthermore, we also have $ab=b^{p^2-p+1}a$ or equivalently $ab^{p+1}=ba$ (as $a^p=b^p$) and therefore using System \eqref{s41}, we obtain    
 \begin{equation}\label{e43}
 h_{a,b}:=b^{j_1}a^{i_1}(b^{j_2}a^{i_2})^{p+1}=b^{j_2}a^{i_2}b^{j_1}a^{i_1}=:g_{a,b}.
 \end{equation}

We now compute $h_{a,b}$.
Notice that 
\[
h_{a,b} = b^{j_1}a^{i_1}(b^{j_2}a^{i_2})^{p+1} 
= (b^{j_2}a^{i_2})^{p} b^{j_1}a^{i_1}b^{j_2}a^{i_2},
\]
as $G^{p} := \{x^{p} \mid x \in G\} \subset \mathbb{Z}(G)$,
where $ \mathbb{Z}(G) $ denotes the center of $ G $, 
and the inclusion holds because $a^{p} = b^{p}$. Substituting the expressions for $ i_1$ and $ j_1 $ and using again the fact that $G^{p} \subset \mathbb{Z}(G) $, we obtain  
\[
h_{a,b}=(b^{j_2}a^{i_2})^{p} b^{j_1}a^{i_1}b^{j_2}a^{i_2}
= (b^{j_2}a^{i_2})^{p}b^{j_2-r}a^{i_2+r}b^{j_2}b^{pt_2}a^{i_2}a^{pt_1}.
\]

Further we employ the relation $ab=b^{p^2-p+1}a=b^{-p+1}a$ repeatedly to obtain $$a^rb^{j_2}=b^{j_2(-p+1)^{r}}a^r.$$Hence, 
\begin{equation*}
\begin{split}
h_{a,b}= (b^{j_2}a^{i_2})^{p}b^{j_2-r}a^{i_2+r}b^{j_2}b^{pt_2}a^{i_2}a^{pt_1}={(b^{j_2}a^{i_2})^{p}b^{j_2-r}a^{i_2}b^{j_2(-p+1)^{r}}b^{pt_2}a^{i_2+r+pt_1}}
\end{split}
\end{equation*}
as $b^{pt_2}\in \mathbb{Z}(G)$.

Moreover, observe that
\[
b^{j_2(-p+1)^r - j_2} = b^{\, j_2 \displaystyle\sum_{k=1}^{r}\binom{r}{k}(-p)^k} \in \mathbb{Z}(G),
\]
since $ p \Bigm| j_2 \left(\displaystyle\sum_{k=1}^{r}\binom{r}{k}(-p)^k\right)$ and $ G^{p} \subset \mathbb{Z}(G)$.
This allows us to rewrite
\[
\begin{aligned}
h_{a,b}
&= (b^{j_2}a^{i_2})^{p} b^{j_2-r} a^{i_2} b^{j_2(-p+1)^{r}} b^{-j_2}b^{j_2}b^{pt_2} a^{i_2+r+pt_1} \\
&= (b^{j_2}a^{i_2})^{p} b^{j_2-r} b^{j_2(-p+1)^r - j_2} a^{i_2} b^{j_2} b^{pt_2} a^{i_2+r+pt_1} \\
&= (b^{j_2}a^{i_2})^{p} b^{j_2-r} b^{j_2(-p+1)^r - j_2} a^{i_2} b^{r} b^{pt_2 + j_2 - r} a^{i_2+r+pt_1}.
\end{aligned}
\]
Since \( a^{i_2}b^{r} = b^{r(-p+1)^{i_2}}a^{i_2} \), therefore
\[
\begin{aligned}
h_{a,b}
&= (b^{j_2}a^{i_2})^{p} b^{j_2-r} b^{j_2(-p+1)^r - j_2} b^{r(-p+1)^{i_2}} a^{i_2} b^{pt_2 + j_2 - r} a^{i_2+r+pt_1} \\
&= (b^{j_2}a^{i_2})^{p} b^{j_2(-p+1)^r - j_2} b^{r(-p+1)^{i_2} - r} b^{j_2} a^{i_2} b^{j_1} a^{i_1}.
\end{aligned}
\]
As $$(b^{j_2}a^{i_2})^{p}=b^{j_2N'+i_2p},~ j_2(-p+1)^r-j_2 =j_2 \left(\sum_{k=1}^{r} \binom{r}{k}(-p)^k\right) \text{ and } r(-p+1)^{i_2}-r=r\left(\sum_{k=1}^{i_2}\binom{i_2}{k}(-p)^k\right),$$
consequently,
\[
h_{a,b}={(b^{j_2}a^{i_2})^{p}b^{j_2(-p+1)^r-j_2}b^{r(-p+1)^{i_2}-r}b^{j_2}a^{i_2}b^{j_1}a^{i_1}}=b^{u'}g_{a,b,}
\]
where 
\[
u'=j_2N'+i_2p+j_2\left(\displaystyle\sum_{k=1}^{r}\binom{r}{k}(-p)^k\right)+r\left(\displaystyle\sum_{k=1}^{i_2}\binom{i_2}{k}(-p)^k\right).
\]

In Equation \eqref{e43}, we have $h_{a,b}=g_{a,b}$, thus it follows that 
\[j_2N'+i_2p+j_2 \left(\sum_{k=1}^{r}\binom{r}{k}(-p)^k\right)+r\left(\sum_{k=1}^{i_2}\binom{i_2}{k}(-p)^k\right) \equiv j_2N'+i_2p-rp(j_2+i_2) \equiv 0 \pmod {p^2}.\] 
Moreover, $$j_2N'+i_2p-rp(j_2+i_2) \equiv j_2p+i_2p-rp(j_2+i_2) \equiv 0 \pmod {p^2}$$ 
as $N' \equiv p \pmod {p^2}.$
It is straightforward to see that the congruence
\[
j_2p + i_2p - rp(j_2 + i_2) \equiv 0 \pmod{p^2}
\]
holds if and only if
$
(1 - r)(j_2 + i_2) \equiv 0 \pmod{p}.
$
Since \( (j_2 + i_2) \not\equiv 0 \pmod{p} \), it follows that \( 1 - r \equiv 0 \pmod{p} \), i.e., \( r = 1 \).
Hence, \( i_2 \equiv i_1 - 1 \pmod{p} \) and \( j_2 \equiv j_1 + 1 \pmod{p} \), as required.

Conversely, assume that \begin{equation*}
\begin{split}
& j_1+i_1 \not \equiv  0  \pmod p,\\& 
  j_2 +i_2\not \equiv  0 \pmod p,
\end{split}
\end{equation*}
 $i_2\equiv i_1-1 \pmod  p$ and $j_2 \equiv j_1+1 \pmod  p$. We have to show that $\mathcal{N}(i_1,j_1,i_2,j_2) \neq \emptyset$.  
For any $u \in \{0,1,\ldots,p-1\}$, $v \in \{0,1,\ldots,p^2-1\}$ and fixed $ 0 \leq t \leq p^3-1$, consider 
 \begin{equation}\label{tempe3}
 g(c_{v+up^2}):=\alpha^{v}\beta^{u}(c_t),
 \end{equation}
 where $\alpha:=b^{j_1}a^{i_1}$ and $\beta:=b^{j_2}a^{i_2}$. We shall show that $g$ is a permutation and $g \in \mathcal{N}(i_1,j_1,i_2,j_2)$. From the above discussion, we have $(b^{j_1}a^{i_1})^p=b^{j_1N+i_1p}$ and $ (b^{j_2}a^{i_2})^p=b^{j_2N'+i_2p}$, where 
$$
\begin{aligned}
N=&\begin{cases} \dfrac{(p^2-p+1)^{i_1p}-1}{(p^2-p+1)^{i_1}-1} & \text{ if } i_1 \neq 0,\\p 
 & \text{ if } i_1=0, \end{cases}\\
N'=&\begin{cases} \dfrac{(p^2-p+1)^{i_2p}-1}{(p^2-p+1)^{i_2}-1} & \text{ if } i_2 \neq 0,\\p 
 & \text{ if } i_2=0\end{cases}
 \end{aligned}$$
  and $N \equiv p \equiv N' \pmod {p^2}$. We first show that $|\alpha|=p^2=|\beta|$.
  Suppose, for contradiction, that $|\alpha|\neq p^{2}$. Then $\alpha^{p}=I$ as $|\alpha| \Bigm| p^2$, and hence
\[
(b^{j_1}a^{i_1})^{p}=b^{\,j_1N+i_1p}=I.
\]
This gives $j_1N+i_1p \equiv 0 \pmod{p^{2}}$, and since $N \equiv p \pmod{p^{2}}$,
this is equivalent to
$
j_1p+i_1p \equiv 0 \pmod{p^{2}}.
$
The congruence $j_1p+i_1p \equiv 0 \pmod{p^{2}}$ then yields
\[
j_1+i_1 \equiv 0 \pmod{p},
\]
which is a contradiction. Similarly, we can show $|\beta|=p^2$. Now we prove that $\alpha^p=\beta^p$. Since $i_2\equiv i_1-1 \pmod  p$ and $j_2 \equiv j_1+1 \pmod  p$, we have $i_2+j_2 \equiv i_1+j_1 \pmod  p$, which gives  $i_2p+j_2p \equiv i_1p+j_1p \pmod {p^2}$. Using $N \equiv p \equiv N' \pmod {p^2}$, it follows that  $j_1N+i_1p \equiv j_2N'+i_2p \pmod {p^2}$. Therefore, 
$$\alpha^p=(b^{j_1}a^{i_1})^p=b^{j_1N+i_1p}=b^{j_2N'+i_2p}=(b^{j_2}a^{i_2})^p=\beta^p.$$
Next, we prove that $\alpha \beta=\beta^{p^2-p+1}\alpha$. Since $i_2\equiv i_1-1 \pmod  p$ and $j_2 \equiv j_1+1 \pmod  p$, we may write $i_2=i_1-1+pt_1$ and $j_2=j_1+1+pt_2$ for some integers $t_1,t_2$. Now,
  \[
  \begin{aligned}
  \alpha \beta=&b^{j_1}a^{i_1}b^{j_1+1+pt_2}a^{i_1-1+pt_1}
  \\=&b^{j_1+pt_2}a^{i_1}ba^{pt_1}b^{j_1}a^{i_1-1}
    \end{aligned}
  \]
 as $b^{pt_2},a^{pt_1} \in \mathbb{Z}(G)$. Using the relation $ab=b^{p^2-p+1}a=b^{-p+1}a$, we obtain $a^{i_1}b=b^{(-p+1)^{i_1}}a^{i_1}$. Hence, 
 \[
 \alpha \beta=b^{j_1+pt_2}b^{(-p+1)^{i_1}}a^{i_1}a^{pt_1}b^{j_1}a^{i_1-1}=b^{j_1+pt_2}b^{(-p+1)^{i_1}}a^{i_1+pt_1-1}ab^{j_1}a^{i_1-1}.
 \]
 Again, from $ab=b^{-p+1}a$, we have $ab^{j_1}=b^{j_1(-p+1)}a$, which implies 
 \begin{equation*}
     \alpha \beta=b^{j_1+pt_2}b^{(-p+1)^{i_1}}a^{i_1+pt_1-1}b^{j_1(-p+1)}a^{i_1}.
 \end{equation*}
 Using $(-p+1)^{i_1} \equiv -pi_1+1 \pmod {p^2}$, $|b|=p^2$ and $b^{-j_1p} \in \mathbb{Z}(G)$,
\begin{equation}\label{tempe2}
     \alpha \beta=b^{-p(i_1+j_1)}b^{j_1+1+pt_2}a^{i_1+pt_1-1}b^{j_1}a^{i_1}=b^{-p(i_1+j_1)}\beta \alpha.
 \end{equation}
 The congruence  $i_1+j_1 \equiv i_2+j_2 \pmod  p$ implies $p(i_1+j_1) \equiv p(i_2+j_2) \pmod {p^2}$ and therefore $$b^{-p(i_1+j_1)}=b^{-p(i_2+j_2)}=(b^{j_2N'+pi_2})^{-1}=(b^{j_2}a^{i_2})^{-p}=\beta^{-p}=\beta^{p^2-p}$$ as $N'\equiv p \pmod {p^2}$. Substituting this into \eqref{tempe2} gives $$\alpha \beta=\beta^{p^2-p+1}\alpha.$$ Thus, $G=\{\alpha^{v}\beta^{u} \mid 0 \leq u \leq p-1, 0 \leq v \leq p^2-1\}$. In Theorem~\ref{T31}, we established that every non-identity permutation in $G$
has no fixed point in $\mathbb{F}_q$. Consequently, no two distinct
permutations in $G$ can take the same value at any element of $\mathbb{F}_q$,
as $G$ is a subgroup of $\mathfrak{S}_q$. Therefore,
\[
\{\alpha^{v}\beta^{u}(c_t)\mid 0 \le u \le p-1,\; 0 \le v \le p^{2}-1\}
 = \mathbb{F}_q,
\]
which shows that $g$ is indeed a permutation.

 Now we will show that $g \in \mathcal{N}(i_1,j_1,i_2,j_2)$, that is, 
  \begin{equation}\label{s42}
  \begin{split}
  gag^{-1}=b^{j_1}a^{i_1}=\alpha \\
  gbg^{-1}=b^{j_2}a^{i_2}=\beta.
  \end{split}
  \end{equation} 
  Since $g$ is a permutation and $\{0,1,\ldots,p^3-1\}=\{v+up^2 \mid 0 \leq u \leq p-1, 0 \leq v \leq p^2-1\}$, it suffices to verify both equations in System \eqref{s42} for all the elements of the form $g(c_{v+up^2})$, where $u \in \{0,1,\ldots,p-1\}$ and $v \in \{0,1,\ldots,p^2-1\}$.  From Equation \eqref{e21}, we have $a(c_{v+up^2})=c_{up^2+(v+1) \pmod {p^2}}$ thus, 
  $$gag^{-1}(g(c_{v+up^2}))=ga(c_{v+up^2})=g(c_{up^2+(v+1) \pmod {p^2}}).$$ 
Further, Equation \eqref{tempe3} implies that $g(c_{v+up^2})=\alpha^{v}\beta^{u}(c_t)$ and $$g(c_{(v+1) \pmod {p^2}+up^2})=\alpha^{(v+1) \pmod {p^2}}\beta^{u}(c_t)=\alpha^{v+1}\beta^{u}(c_t)$$ as $|\alpha|=p^2$.
Thus,
 \[
 \alpha(g(c_{v+up^2}))= \alpha(\alpha^{v}\beta^{u}(c_t))=\alpha^{v+1}\beta^{u}(c_t)=g(c_{(v+1) \pmod {p^2}+up^2})=gag^{-1}(g(c_{v+up^2})).
 \]
 Hence, $g$ satisfies the first equation of the System \eqref{s42}. 
 
 For the second equation of the System \eqref{s42}, assume that $v \equiv k \pmod  p$ for some $k \in \{0,1,\ldots,p-1\}$, which immediately yields  $v+up^2 \equiv k \pmod  p$ and thus $b(c_{v+up^2})=b_k^{1+kp}(c_{v+up^2})$ as from the definition of $b$, $b=\displaystyle \prod_{\ell=0}^{p-1}b_{\ell}^{1+\ell p}$ and Equation \eqref{e22} implies $b_{\ell}(c_t)=c_t$ whenever $t \not \equiv \ell \pmod  p$. 
 Since $a^p=b^p$, we get $a^{kp}(c_{v+up^2})=b^{kp}(c_{v+up^2})=b_k^{kp(1+kp)}(c_{v+up^2})=b_k^{kp}(c_{v+up^2}).$  Furthermore, using Equation \eqref{tempe4}, 
 $$a^{kp}(c_{v+up^2})=c_{up^2+(v+kp)\pmod {p^2}}$$
as $up^2\leq v+up^2 \leq p^2-1+up^2 =(u+1)p^2-1$.

Thus, 
\begin{equation}\label{R43}
b(c_{v+up^2})=b_k^{1+kp}(c_{v+up^2})=b_k(a^{kp}(c_{v+up^2}))=b_k(c_{up^2+(v+kp)\pmod {p^2}}).
\end{equation}
Now we verify the second equation of System \eqref{s42} into two cases depending on the value of $u$.
 
 \textbf{Case 1:} In this case, we consider $u<p-1$. Note that $up^2+(v+kp)\pmod {p^2} \leq (p-2)p^2+p^2-1=p^3-p^2-1<p^3-p^2.$ Therefore using Equation \eqref{e22} in Equation \eqref{R43}, we have 
 \[
 b(c_{v+up^2})=b_k^{1+kp}(c_{v+up^2})=b_{k}(c_{up^2+(v+kp)\pmod {p^2}})=c_{(u+1)p^2+(v+kp)\pmod {p^2}}.
 \]
 This gives 
  \[
  gbg^{-1}(g(c_{v+up^2})) = g(b(c_{v+up^2}))=g(c_{(u+1)p^2+(v+kp)\pmod {p^2}}).
 \]
Moreover, from the definition of $g$ in Equation \eqref{tempe3}, we obtain
\[
g(c_{v+up^2})=\alpha^{v}\beta^{u}(c_t) \text{ and } g(c_{(u+1)p^2+(v+kp)\pmod {p^2}})=\alpha^{(v+kp)\pmod {p^2}}\beta^{u+1}(c_t)=\alpha^{v+kp}\beta^{u+1}(c_t), 
\]
as $|\alpha|=p^2$.
Using $\alpha \beta=\beta^{p^2-p+1}\alpha$, $|\beta|=p^2$ and $\alpha^p=\beta^p$, we can easily get $\beta \alpha=\alpha \beta^{p+1}$. The equality $\beta \alpha=\alpha \beta^{p+1}$ gives $\beta \alpha^v=\alpha^v\beta^{(p+1)^v}=\alpha^v\beta^{vp+1}$ as $(p+1)^v \equiv 1+vp \pmod {p^2}$ and $|\beta|=p^2$. Therefore,
\begin{equation*}
\beta(g(c_{v+up^2}))=\beta(\alpha^{v}\beta^{u}(c_t)) = \alpha^v \beta^{vp+1}\beta^{u}(c_t)=\alpha^{v+vp}\beta^{u+1}(c_t)=\alpha^{v+kp}\beta^{u+1}(c_t)
\end{equation*}
as $\alpha^p=\beta^p$ and $vp \equiv kp \pmod {p^2}$. Thus
 \[ gbg^{-1}(g(c_{v+up^2}))=g(c_{(u+1)p^2+(v+kp)\pmod {p^2}})=\alpha^{v+kp}\beta^{u+1}(c_t)=\beta(g(c_{v+up^2})).\]

\textbf{Case 2:} When $u=p-1$. We divide this case into following two subcases.

\textbf{Subcase 2.1:} In this subcase, we consider $0 \leq (v+kp)\pmod {p^2}<p^2-p$. So, we have $p^3-p^2 \leq up^2+(v+kp)\pmod {p^2} <p^3-p$. By applying Equation \eqref{e22} to \eqref{R43}, we have 
\[
b(c_{v+up^2})=b_k^{1+kp}(c_{v+up^2})=b_k(c_{up^2+(v+kp)\pmod {p^2}})=c_{p+(v+kp)\pmod {p^2}}.
\]
Thus, 
  \[
  gbg^{-1}(g(c_{v+up^2}))= g(b(c_{v+up^2}))=g(c_{p+(v+kp)\pmod {p^2}}).
 \]
From the definition of $g$,
\[
gbg^{-1}(g(c_{v+up^2}))=g(c_{p+(v+kp)\pmod {p^2}})=\alpha^{p+v+kp}(c_t), 
\]
as $0<p+(v+kp)\pmod {p^2}<p+p^2-p= p^2$, that is, $0<p+(v+kp)\pmod {p^2} \leq p^2-1$.
On the right side of the second equation in System \eqref{s42},
\[
\beta(g(c_{v+up^2}))=\beta(\alpha^{v}\beta^{u}(c_t))= \alpha^v \beta^{vp+1} \beta^{u}(c_t)=\alpha^{v+vp+p}(c_t)=\alpha^{v+kp+p}(c_t)
\]
as $g(c_{v+up^2})=\alpha^{v}\beta^{u}(c_t)$, $vp \equiv kp \pmod {p^2}$, $\alpha^p=\beta^p$ and $u+1=p$. Therefore, $gbg^{-1}(g(c_{v+up^2}))=\beta(g(c_{v+up^2}))$.

\textbf{Subcase 2.2:} In this subcase, we assume that $p^2-p \leq (v+kp)\pmod {p^2} <p^2$. This subcase can be processed similarly to Subcase 2.1. 

From the above cases, we conclude that $g \in \mathcal{N}(i_1,j_1,i_2,j_2)$ and so $\mathcal{N}(i_1,j_1,i_2,j_2) \neq \emptyset$. Observe that, for fixed $\alpha = b^{j_1}a^{i_1}$ and $\beta = b^{j_2}a^{i_2}$, the function $g$ is completely determined by the choice of $c_t$, and there are $p^{3}$ possible choices for $c_t$. Moreover, from Equation \eqref{tempe3}, we have $g(c_0)=c_t$, which ensures that distinct choices of $c_t$ produce distinct functions $g$. Hence,
$|\mathcal{N}(i_1,j_1,i_2,j_2)| \ge p^{3}.$
Finally, using Lemma \ref{Atmost_lemma_T31}, we have $|\mathcal{N}(i_1,j_1,i_2,j_2)|=p^3$.
\end{proof}

\begin{thm}\label{Enumeration_T31}
Let $f$ be a permutation group polynomial corresponding to the group $G$ constructed in Theorem \ref{T31} and $\mathcal{P}_{f}$ be the number of permutation group polynomials of the form as $f$. Then
\[
\mathcal{P}_{f}=\dfrac{(p^3!)^2}{p^3(p^4-p^3)}.
\] 
\end{thm}
\begin{proof}
Let $g(X_1, X_2)$ be a permutation group polynomial of the form as $f$, and let $H$ be its corresponding group. Then from Remark \ref{R_conjugate},  we have $H = hGh^{-1}$ for some $h \in \mathfrak{S}_q$. Also, every conjugate $H$ of $G$ in $\mathfrak{S}_q$ gives the permutation group polynomials of the form $f$. Furthermore, we can get $p^3!$ permutation group polynomials of the same form from $H$ by permuting the elements of $H$. Therefore,
\[
\mathcal{P}_{f}= p^3!|Conj_{\mathfrak{S}_q}(G)|=\dfrac{(p^3!)^2}{|\mathfrak{N}_{\mathfrak{S}_q}(G)|},
\]
where $Conj_{\mathfrak{S}_q}(G)=\{hGh^{-1} \mid h \in \mathfrak{S}_q\}$ is the set of all conjugates of $G$ in $\mathfrak{S}_q$ and $\mathfrak{N}_{\mathfrak{S}_q}(G)=\{h \in \mathfrak{S}_q \mid hGh^{-1}=G\}$ is the normalizer of $G$ in $\mathfrak{S}_q$. Thus, it is sufficient to find $|\mathfrak{N}_{\mathfrak{S}_q}(G)|$. As $G=\langle a,b \rangle$, $\mathfrak{N}_{\mathfrak{S}_q}(G)=\{h\in \mathfrak{S}_q \mid hah^{-1}=b^{j_1}a^{i_1}, hbh^{-1}=b^{j_2}a^{i_2}, 0 \leq i_1,i_2 \leq p-1, 0 \leq j_1,j_2 \leq p^2-1\}$.
Therefore,
\[
\mathfrak{N}_{\mathfrak{S}_q}(G)=\displaystyle \bigcup\mathcal{N}(i_1,j_1,i_2,j_2),
\]
where $i_1,i_2 \in \{0,1,\ldots,p-1\}, \text{ } j_1,j_2 \in \{0,1,\ldots,p^2-1\}$.

It is easy to see that $\mathcal{N}(i_1,j_1,i_2,j_2) \cap \mathcal{N}(i'_1,j'_1,i'_2,j'_2)=\emptyset$ for $(i_1,j_1,i_2,j_2) \neq (i'_1,j'_1,i'_2,j'_2)$.  Using this fact and Lemma \ref{Atleast_lemma_T31},
\[
|\mathfrak{N}_{\mathfrak{S}_q}(G)|
= \sum_{\substack{
i_1,i_2 \in \{0,1,\ldots,p-1\} \\
j_1,j_2 \in \{0,1,\ldots,p^2-1\}
}}
|\mathcal{N}(i_1,j_1,i_2,j_2)|.
\]
where 
$
 j_1+i_1 \not \equiv  0  \pmod  p,~
  j_2+i_2 \not \equiv  0 \pmod  p,
$ $i_2\equiv i_1-1 \pmod  p$ and $j_2 \equiv j_1+1 \pmod  p$. Lemma \ref{Atleast_lemma_T31} also implies that whenever $\mathcal{N}(i_1,j_1,i_2,j_2) \neq \emptyset$, $|\mathcal{N}(i_1,j_1,i_2,j_2)|=p^3$. Thus, we have 
\[
|\mathfrak{N}_{\mathfrak{S}_q}(G)|=\displaystyle p^3 |A|, 
\]
where $A=\{(i_1,j_1,i_2,j_2) \mid  j_1+i_1 \not \equiv  0  \pmod  p,~
  j_2+i_2 \not \equiv  0 \pmod  p, ~i_2\equiv i_1-1 \pmod  p, ~j_2 \equiv j_1+1 \pmod  p, ~0 \leq i_1,i_2 \leq p-1, ~0 \leq  j_1,j_2 \leq p^2-1\}$. Let us compute the cardinality of $A$. Notice that for a given $0 \leq i_1 \leq p-1$, there are exactly $p^2-p$ choices of $j_1$ such that  $j_1+i_1 \not \equiv 0  \pmod  p$. The conditions $j_1+i_1 \not \equiv 0  \pmod  p$, $i_2\equiv i_1-1 \pmod  p$ and $j_2 \equiv j_1+1 \pmod  p$ together imply that $j_2+i_2 \not \equiv  0 \pmod  p$. It is easy to see that for a given tuple $(i_1,j_1)$, the choices of $(i_2,j_2)$ such that $i_2\equiv i_1-1 \pmod  p,~ j_2 \equiv j_1+1 \pmod  p$ are exactly $p$. Thus, $|A|=p(p^2-p)p=p^4-p^3$ and $|\mathfrak{N}_{\mathfrak{S}_q}(G)|=p^3(p^4-p^3)$. Consequently, 
\[
\mathcal{P}_{f}=\dfrac{(p^3!)^2}{|\mathfrak{N}_{\mathfrak{S}_q}(G)|}=\dfrac{(p^3!)^2}{p^3(p^4-p^3)}.
\]
\end{proof}
\begin{prop}
Let $f$ be a permutation group polynomial constructed in Theorem \ref{T31} and its corresponding permutation polynomial tuple is $\underline{\beta}_f=(\beta_0,\beta_1,\ldots,\beta_{q-1})\in \mathfrak{S}_q^{q}$. Then there are $q!$ permutation group polynomials equivalent to $f$.
\end{prop}
\begin{proof}
It is easy to verify that for $G$ defined in Theorem \ref{T31}, $\mathfrak{C}_{\mathfrak{S}_q}(G)=\mathcal{N}(1,0,0,1)$.  Moreover, from Lemma \ref{Atleast_lemma_T31}, we have $|\mathcal{N}(1,0,0,1)|=p^n=q$, thus $|\mathfrak{C}_{\mathfrak{S}_q}(G)|=q$. Thus, the result follows from Lemma \ref{Equivalent_PGP}. 
  \end{proof}
\begin{lem}\label{Atmost_lemma_HK1}
Let $q=2^m$, where $m \geq 3$ is an integer, and $G=\langle a,b \rangle $ be the subgroup of $\mathfrak{S}_q$ defined in Lemma \ref{HK1}. Moreover, let $i_1,i_2 \in  \{0,1\}$, $j_1,j_2 \in \{0,1,\ldots,\frac{q}{2}-1\}$ and $\mathcal{N}(i_1,j_1,i_2,j_2)=\{h \in \mathfrak{S}_q \mid hah^{-1}=b^{j_1}a^{i_1}, hbh^{-1}=b^{j_2}a^{i_2}\}$.  Then  $|\mathcal{N}(i_1,j_1,i_2,j_2)| \leq  2^m$.
\end{lem}
\begin{proof}
The proof follows along the similar idea as in Lemma \ref{Atmost_lemma_T31}.
\end{proof}
\begin{lem}\label{Atleast_lemma_HK1}
Let $q=2^m$, where $m \geq 3$ is an integer, and $G=\langle a,b \rangle$ be the subgroup of $\mathfrak{S}_q$ defined as in Lemma \ref{HK1}. Furthermore,  let $i_1,i_2 \in \{0,1\}$ and $j_1,j_2 \in \{0,1,\ldots,\frac{q}{2}-1\}$. Then $\mathcal{N}(i_1,j_1,i_2,j_2) \neq \emptyset$ if and only if $i_2=0$, $\gcd(j_2,\frac{q}{2})=1$ and $i_1=1$, where $\mathcal{N}(i_1,j_1,i_2,j_2)$ is as defined as in Lemma \ref{Atmost_lemma_HK1}. In this case, $|\mathcal{N}(i_1,j_1,i_2,j_2)|=2^m$.
\end{lem}
\begin{proof} We first assume that $\mathcal{N}(i_1,j_1,i_2,j_2) \neq \emptyset$ for some $i_1,i_2 \in \{0,1\}$ and $j_1,j_2 \in \{0,1,\ldots,\frac{q}{2}-1\}$. We have to show that $i_2=0$, $\gcd(j_2,\frac{q}{2})=1$ and $i_1=1$.  Let $h \in \mathcal{N}(i_1,j_1,i_2,j_2)$, thus,  we obtain 
\begin{equation}\label{s43}
\begin{split}
&hah^{-1}=b^{j_1}a^{i_1}:=\alpha\\
&hbh^{-1}=b^{j_2}a^{i_2}:=\beta.
\end{split}
\end{equation}
On the contrary assume that $i_2=1$, i.e., $\beta=b^{j_2}a$. Then using the relation $ab=b^{-1}a$ and the fact $|a|=2$ given in Lemma \ref{HK1}, we have $$b^{j_2}ab^{j_2}a=b^{j_2}b^{-j_2}a^2=I,$$ which implies $|\beta|=|b^{j_2}a| \leq 2$. Further, using the second equation in System \eqref{s43}, we obtain $|\beta|=|b|$, which gives $|b|\leq 2$. This is a contradiction as $|b| = \frac{q}{2} = 2^{m-1}$ stated in Lemma~\ref{HK1}. Thus, $i_2=0$ or equivalently, $\beta=b^{j_2}$ for some $j_2 \in \{0,1,\ldots,\frac{q}{2}-1\}$.  Next, if we assume, on the contrary, that $\gcd(j_2,\frac{q}{2}=2^{m-1})>1$. Then $|\beta|=|b^{j_2}|< 2^{m-1}.$ This again leads to a contradiction as $|\beta|=|b| = \frac{q}{2} = 2^{m-1}$.  Thus, $\gcd(j_2,\frac{q}{2})=1$ and $\beta=b^{j_2}$. Now, it remains to show that $i_1=1$. Suppose if possible $i_1=0$, then $\alpha=b^{j_1}$ for some $j_1 \in \{0,1,\ldots,\frac{q}{2}-1\}$. Since $\gcd(j_2,\frac{q}{2})=1$, the congruence $j_2X \equiv j_1 \pmod {\frac{q}{2}}$ has a solution in $\{0,1, \ldots, \frac{q}{2}-1\}$, say $k$. By using $|b| = \frac{q}{2}$ and $k j_2 \equiv j_1 \pmod {\frac{q}{2}}$, it follows that $b^{j_1} = b^{k j_2}$, i.e., $\alpha=\beta^{k}$. Raising the second equation of System~\ref{s43} to the $k$-th power gives 
$\beta^{k} = h b^{k} h^{-1}$. 
Since $\alpha = \beta^{k}$, we obtain $h a h^{-1} = h b^{k} h^{-1}$  and so
$a = b^{k}$, which is impossible because $a \notin \langle b \rangle$, as $G = \{ b^{j} a^{i} : 0 \le j \le \frac{q}{2}-1,\; 0 \le i \le 1 \}$ has order $q$.

Conversely, we assume that $i_2=0$, $\gcd(j_2,\frac{q}{2})=1$ and $i_1=1$. We have to show that $\mathcal{N}(i_1,j_1,i_2,j_2) \neq \emptyset$. Let $\alpha=b^{j_1}a$, $\beta=b^{j_2}$. It is easy to verify that $A_1\cup A_2 \cup A_3 \cup A_4=\{0,1,\ldots,q-1\}$ and $A_{\ell}\cap A_{\ell'}=\emptyset$ for all $1 \leq \ell, \ell' \leq 4$ such that $\ell \neq \ell'$, where
\[
\begin{aligned}
    A_1=&\left\{\frac{q}{2}-2(n_1+1) \mid 0\leq  n_1 \leq \frac{q}{4}-1\right\},\\
    A_2=& \left\{\frac{q}{2}+2n_1 \mid 0 \leq n_1 \leq \frac{q}{4}-1\right\},\\
    A_3=&\left\{ 1+2n_1 \mid 0\leq n_1 \leq \frac{q}{4}-1 \right\}, ~\text{ and }\\
    A_4=&\left\{q-(1+2n_1) \mid 0 \leq n_1 \leq \frac{q}{4}-1 \right\}.
\end{aligned}
\]
Now for any fixed integer $ 0 \leq t \leq q-1$, consider the following function
\begin{equation}\label{tempe1}
\begin{cases}
	g(c_{\frac{q}{2}-2(n_1+1)})= \beta^{n_1}(c_{t})~&~\mbox{if}~0\leq  n_1 \leq \frac{q}{4}-1,\\
    g(c_{\frac{q}{2}+2n_1})= \beta^{n_1+\frac{q}{4}}(c_{t})~&~\mbox{if}~ 0 \leq n_1 \leq \frac{q}{4}-1,\\
    g(c_{1+2n_1})=\alpha \beta^{\frac{q}{4}-1-n_1}(c_{t})~&~\mbox{if}~0\leq n_1 \leq \frac{q}{4}-1,\\
    g(c_{q-(1+2n_1)})= \alpha \beta^{\frac{q}{2}-1-n_1}(c_{t})~&~\mbox{if}~0 \leq n_1 \leq \frac{q}{4}-1.
\end{cases}
\end{equation}
First, we show that $g$ is a permutation. As $a \not \in \langle b \rangle$, $\alpha =b^{j_1}a\neq I$. Next, $\alpha^2=b^{j_1}ab^{j_1}a=b^{j_1}b^{-j_1}a^2=I$ as $ab=b^{-1}a$ and $|a|=2$. This gives $|\alpha|=2$. Since $|b|=\frac{q}{2}$ and $\gcd(j_2,\frac{q}{2})=1$, thus $|\beta|=|b^{j_2}|=\frac{q}{2}=2^{m-1}$. Now, $$\alpha \beta=b^{j_1}ab^{j_2}=b^{j_1}b^{-j_2}a=b^{-j_2}b^{j_1}a=\beta^{-1}\alpha.$$ Hence $G=\langle \alpha, \beta \rangle=\{\alpha^{i}\beta^{j} \mid 0 \leq i \leq 1, 0 \leq j \leq \frac{q}{2}\}=\{\beta^{n_1} \mid 0 \leq n_1 \leq \frac{q}{4}-1\} \bigcup \{\beta^{n_1+\frac{q}{4}}\mid 0 \leq n_1 \leq \frac{q}{4}-1\} \bigcup \{\alpha \beta^{\frac{q}{4}-1-n_1}\mid 0 \leq n_1 \leq \frac{q}{4}-1\} \bigcup \{\alpha \beta^{\frac{q}{2}-1-n_1} \mid 0 \leq n_1 \leq \frac{q}{4}-1\}$. In Lemma \ref{HK1}, it is stated that every permutation, except the identity permutation, of $G$ has no fixed point. However, this fact is true if and only if any two distinct permutations of $G$ cannot have same image at any element of $\F_q$. This gives that $B_{\ell} \cap B_{\ell'}=\emptyset$ for all $1 \leq \ell, \ell' \leq 4$ such that $\ell \neq \ell'$, where 
$$
\begin{aligned}
B_1&=\left\{\beta^{n_1} (c_t) \mid 0 \leq n_1 \leq \frac{q}{4}-1\right\}, \quad B_2= \left\{\beta^{n_1+\frac{q}{4}}(c_t)\mid 0 \leq n_1 \leq \frac{q}{4}-1\right\},\\
B_3&=\left\{\alpha \beta^{\frac{q}{4}-1-n_1}(c_t)\mid 0 \leq n_1 \leq \frac{q}{4}-1\right\},  \text{ and } B_4=\left\{\alpha \beta^{\frac{q}{2}-1-n_1}(c_t) \mid 0 \leq n_1 \leq \frac{q}{4}-1\right\}.
\end{aligned}
$$
Thus, $B_1\cup B_2 \cup B_3 \cup B_4=\F_q$, which implies that $g$ is a permutation. Now we shall show that 
\begin{equation}\label{e44}
gag^{-1}=\alpha,
\end{equation} 
and 
\begin{equation}\label{e45}
gbg^{-1}=\beta.
\end{equation}
Since $g$ is a permutation, $\{g(x) \mid x \in \F_q\}=\F_q$ and thus it suffices  to verify Equations \eqref{e44} and \eqref{e45} for all $g(x)$, where $x \in \F_q$. First we verify Equations \eqref{e44} and \eqref{e45} for the elements of the form $g(c_{\frac{q}{2}-2(n_1+1)})$, where $0\leq  n_1 \leq \frac{q}{4}-1$. 

\textbf{Case 1:}
When $n_1 \neq \frac{q}{4}-1$. Then 
\[
gag^{-1}(g(c_{\frac{q}{2}-2(n_1+1)}))=ga(c_{\frac{q}{2}-2(n_1+1)})= g(c_{\frac{q}{2}-2n_1-1})
\]
as $a(c_{\frac{q}{2}-2(n_1+1)})=c_{\frac{q}{2}-2n_1-1}$ from the cyclic structure of the permutation $a$ defined in Lemma \ref{HK1}. 

Now, from the definition of $g$ in Equation \eqref{tempe1}, it follows that
\[
g(c_{\frac{q}{2}-2(n_1+1)})=\beta^{n_1}(c_t) \text{ and } g(c_{\frac{q}{2}-2n_1-1})=g(c_{2(\frac{q}{4}-n_1-1)+1})=\alpha \beta^{\frac{q}{4}-1-(\frac{q}{4}-n_1-1)}(c_t)=\alpha \beta^{n_1}(c_t).
\]
Thus, $$\alpha (g(c_{\frac{q}{2}-2(n_1+1)}))=\alpha (\beta^{n_1}(c_t))=g(c_{\frac{q}{2}-2n_1-1})=gag^{-1}(g(c_{\frac{q}{2}-2(n_1+1)})).$$


Next, the cyclic structure of the permutation $b$ in Lemma \ref{HK1} implies
\[
b(c_{\frac{q}{2}-2(n_1+1)})=c_{\frac{q}{2}-2(n_1+2)}.
\]
This gives
\[
gbg^{-1}(h(c_{\frac{q}{2}-2(n_1+1)}))=gb(c_{\frac{q}{2}-2(n_1+1)})=g(c_{\frac{q}{2}-2(n_1+2)}). 
\]
Again from the definition of $g$ in Equation \eqref{tempe1}, we have 
\[
g(c_{\frac{q}{2}-2(n_1+2)})= \beta^{n_1+1}(c_t).
\]
Therefore, $\beta (g(c_{\frac{q}{2}-2(n_1+1)}))=\beta(\beta^{n_1}(c_t))=g(c_{\frac{q}{2}-2(n_1+2)})=gbg^{-1}(g(c_{\frac{q}{2}-2(n_1+1)}))$. 

\textbf{Case 2:} Let $n_1=\frac{q}{4}-1$. Then we have, $gag^{-1}(g(c_0))= g(c_1)$ as $a(c_0)=c_1$. Since from Equation \eqref{tempe1}, we have $g(c_0)=\beta^{\frac{q}{4}-1}(c_t)$ and $g(c_1)=\alpha \beta^{\frac{q}{4}-1}(c_t)$, therefore,
\[
\alpha(g(c_0))=\alpha (\beta^{\frac{q}{4}-1}(c_t))=g(c_1)=gag^{-1}(g(c_0)).
\]
 From Equation \eqref{e45}, we have $gbg^{-1}(g(c_0))= g(c_{\frac{q}{2}})$ as $b(c_0)=c_{\frac{q}{2}}$. Further, we have $g(c_{\frac{q}{2}})=\beta^{\frac{q}{4}}(c_t)$, thus,
 \[
 \beta(g(c_0))=\beta(\beta^{\frac{q}{4}-1}(c_t))=g(c_{\frac{q}{2}})=gbg^{-1}(g(c_0)).
 \]
 Similarly, one can prove that $g$ satisfies Equations \eqref{e44} and \eqref{e45} for the remaining elements of $\F_q$. 

Therefore, $g\in \mathcal{N}(i_1,j_1,i_2,j_2)$.  The construction of $g$ depends solely on the choice of $c_t$. 
From \eqref{tempe1} we have $g(c_{\frac{q}{2}-2}) = c_t$. 
Hence, choosing distinct values of $c_t$ yields distinct permutation $g$. 
Therefore,
\[
|\mathcal{N}(i_1,j_1,i_2,j_2)| \ge q = 2^{m}.
\]
Combining this with Lemma~\ref{Atmost_lemma_HK1}, we conclude that
$
|\mathcal{N}(i_1,j_1,i_2,j_2)| = 2^{m}.
$
\end{proof}
\begin{thm}\label{Enumeration_HK1}
Let $f_1$ be a permutation group polynomial corresponding to the group $G$ constructed in Lemma \ref{HK1} and $\mathcal{P}_{f_1}$ be the number of permutation group polynomials of the form of $f_1$. Then 
\[
\mathcal{P}_{f_1}=\frac{(2^m!)^2}{2^{2m-1}\phi(2^{m-1})}.
\]
\end{thm}
\begin{proof}
Using the similar argument as in Theorem \ref{Enumeration_T31}, we have 
\[
\mathcal{P}_{f_1}=\frac{(2^m!)^2}{|\mathfrak{N}_{\mathfrak{S}_q}(G)|}, 
\]
where $\mathfrak{N}_{\mathfrak{S}_q}(G)=\{h \in \mathfrak{S}_q \mid hGh^{-1}=G\}$ is the normalizer of $G$ in $\mathfrak{S}_q$. Again by the  similar reason as in Theorem \ref{Enumeration_T31},
\[
|\mathfrak{N}_{\mathfrak{S}_q}(G)|=\displaystyle \sum_{i_1,i_2 \in \{0,1\}, \text{ } j_1,j_2 \in \{0,1,\ldots,\frac{q}{2}-1\}}|\mathcal{N}(i_1,j_1,i_2,j_2)|,
\]
where $\mathcal{N}(i_1,j_1,i_2,j_2)$ is as defined as in Lemma \ref{Atmost_lemma_HK1}. Further
Lemma \ref{Atleast_lemma_HK1} implies that 
\[
|\mathfrak{N}_{\mathfrak{S}_q}(G)|=2^m|B|,
\]
where $B=\{(i_1,j_1,i_2,j_2)\mid i_1=1,i_2=0,0 \leq j_1,j_2 \leq \frac{q}{2}-1=2^{m-1}-1, (j_2,2^{m-1})=1\}$. Note that $|B|=2^{m-1}\phi(2^{m-1})$. Thus, 
\[
\mathcal{P}_{f_1}=\frac{(2^m!)^2}{2^{2m-1}\phi(2^{m-1})}.
\]
\end{proof}
\begin{prop}\label{Cor_HK1}
Let $f_1$ be a permutation group polynomial constructed in Lemma \ref{HK1} and its corresponding permutation polynomial tuple is $\underline{\beta}_{f_1}=(\beta_0,\beta_1,\ldots,\beta_{q-1})\in \mathfrak{S}_q^{q}$. Then there are $q!$ permutation group polynomials equivalent to $f_1$.
\end{prop}
\begin{proof}
The proof follows from Lemma \ref{Equivalent_PGP} and the fact that the centralizer of $G$, where $G$ defined in Lemma \ref{HK1}, has cardinality $q$ from Lemma \ref{Atleast_lemma_HK1}.
  \end{proof} 
\begin{lem}\label{Lemma_HK2} 
Let $q=4k$ and $G =\langle a,b \rangle$ be the group as in Lemma \ref{HK2}, where $k=2^{\ell}$ for some integer $\ell \geq 2$. Moreover, let $i_1,i_2 \in \{0,1\}$, $j_1, j_2 \in \{0,1,\ldots, \frac{q}{2}-1\}$ and $\mathcal{N}(i_1,j_1,i_2,j_2)=\{h \in \mathfrak{S}_q \mid hah^{-1}=b^{j_1}a^{i_1}, hbh^{-1}=b^{j_2}a^{i_2}\}$. Then  $\mathcal{N}(i_1,j_1,i_2,j_2) \neq \emptyset$ if and only if $(i_1,j_1,i_2,j_2) \in \left\{(0,\frac{q}{4},1,j_2), (1,0,i,j_2), (1,\frac{q}{4},i,j_2) \mid  0 \leq i \leq 1, 0 \leq j_2 \leq \frac{q}{2}-1, ~\gcd(j_2,\frac{q}{2})=1\right\}$. Moreover, in this case, $|\mathcal{N}(i_1,j_1,i_2,j_2)|=q$.
\end{lem}
\begin{proof}
 We first assume that $\mathcal{N}(i_1,j_1,i_2,j_2)\neq \emptyset$. This implies that there exists $h \in \mathfrak{S}_q$ such that $hah^{-1}=b^{j_1}a^{i_1}:=\alpha$ and $hbh^{-1}=b^{j_2}a^{i_2}:=\beta$. Thus, we must have $|b^{j_1}a^{i_1}|=2$ and $|b^{j_2}a^{i_2}|=\frac{q}{2}=2k$. If $i_1=0$, then clearly $j_1=\frac{q}{4}$. For $i_1=1$, note that $b^{j_1}ab^{j_1}a=I$ gives that $b^{j_1(k+2)}=I$ as $ab=b^{k+1}a$, i.e., $b^{2j_1(\frac{k}{2}+1)}=I$, which is possible if and only if $j_1=0$ or $j_1=k=\frac{q}{4}$. Next for $i_2=0$, it is clear that $\gcd(j_2,\frac{q}{2})=1$ as $|b^{j_2}a^{i_2}|=\frac{q}{2}=2k$. Now if we take $i_2=1$, then using relation $ab=b^{k+1}a$, we have 
\begin{equation*}
\begin{split}
(b^{j_2}a)^{k}&=\underbrace{(b^{j_2}ab^{j_2}a)(b^{j_2}ab^{j_2}a)\cdots(b^{j_2}ab^{j_2}a)}_{ \frac{k}{2} \text{ times }}\\
&=\underbrace{(b^{j_2(k+2)})(b^{j_2(k+2)})\cdots(b^{j_2(k+2)})}_{ \frac{k}{2} \text{ times }}=b^{j_2\frac{k}{2}(k+2)}.
\end{split}
\end{equation*}
Thus, $(b^{j_2}a)^{k}=I$ if and only if $2k \mid j_2\frac{k}{2}(k+2)$, that is, $2k \mid j_2k$ as $\gcd(\frac{k}{2}+1,2k)=1$. But $2k \mid j_2k$ if and only if $2 \mid j_2$. Therefore, if $2 \nmid j_2$ or $\gcd(j_2,\frac{q}{2})=1$, then $|b^{j_2}a|=2k=\frac{q}{2}$. Up to this point, we have established that 
\[
(i_1, j_1) \in \left\{ \left(0, \tfrac{q}{4}\right), (1, 0), \left(1, \tfrac{q}{4}\right) \right\},
\text{and} 
 \gcd\!\left(j_2, \tfrac{q}{2}\right) = 1.
\]
Since $\gcd(j_2,\frac{q}{2})=1$, therefore, $b=b^{uj_2}$ for some $u \in \mathbb{Z}$. This implies that for $(i_1, j_1)=\left(0, \tfrac{q}{4}\right)$, then $i_2 \neq 0$, otherwise $b^{\frac{q}{4}}=b^{u'j_2}$ for some $u' \in \mathbb{Z}$, which gives that $a=b^{u'}$, a contradiction. Thus, we have $(i_1,j_1,i_2,j_2) \in \{(0,\frac{q}{4},1,j_2), (1,0,i,j_2), (1,\frac{q}{4},i,j_2) \mid  0 \leq i \leq 1, 0 \leq j_2 \leq \frac{q}{2}-1, ~\gcd(j_2,\frac{q}{2})=1\}$.


Conversely, we assume that $(i_1,j_1,i_2,j_2) \in \{(0,\frac{q}{4},1,j_2), (1,0,i,j_2), (1,\frac{q}{4},i,j_2) \mid  0 \leq i \leq 1, 0 \leq j_2 \leq \frac{q}{2}-1, \gcd(j_2,\frac{q}{2})=1\}$, $\alpha:=b^{j_1}a^{i_1}$ and $\beta:=b^{j_2}a^{i_2}$. For any integer $0 \leq t \leq q-1$, we define the following function $g$ 
\begin{equation}\label{LE48}
    \begin{cases}
       g(c_{2r})= \beta^{2r}(c_{t})& \text{ if } 0 \leq r \leq \frac{k}{2}-1, \\
       g(c_{2r+k})= \beta^{2r+k}(c_{t})& \text{ if } 0 \leq r \leq \frac{k}{2}-2, \\
       g(c_{2(2r+k-1)})= \beta^{2r+1}(c_t)& \text{ if } 0 \leq r \leq \frac{k}{2}-1, \\
       g(c_{2(2r+k)})= \beta^{k+2r+1}(c_{t})&  \text{ if }  0  \leq r \leq \frac{k}{2}-1, \\
       g(c_{2r+1})= \alpha \beta^{2r} (c_{t})&  \text{ if }  0 \leq r \leq \frac{k}{2}-1, \\
       g(c_{2r+1+k})= \alpha \beta^{2r+k} (c_{t})&  \text{ if }  0 \leq r \leq \frac{k}{2}-2, \\
       g(c_{2(2r+k)+1})= \alpha \beta^{k+2r+1} (c_{t})&  \text{ if }  0 \leq r \leq \frac{k}{2}-1, \\
        g(c_{2(2r+k)-1})= \alpha \beta^{2r+1} (c_{t})&  \text{ if }  0  \leq r \leq \frac{k}{2}-1, \\
        g(c_{4k-1})= \beta^{2k-2}(c_{t}), \\
        g(c_{4k-2})= \alpha \beta^{2k-2} (c_t).
	\end{cases}
\end{equation}

We shall show that $g$ is a permutation and $g \in \mathcal{N}(i_1,j_1,i_2,j_2)$. It is clear from the above that $|\alpha|=2$ and $|\beta|=\frac{q}{2}=2k$. First, we prove that $\alpha \beta=\beta^{k+1}\alpha$ in all cases. Begin with the case $(i_1,j_1)=(1,0)$, that is, $\alpha=a$. If $\beta=b^{j_2}$, where $\gcd(j_2,\frac{q}{2})=1$, then it is clear that $\alpha \beta =ab^{j_2}=b^{j_2(k+1)}a=\beta^{k+1}\alpha$. Now if $\beta=b^{j_2}a$, then $\alpha \beta=a b^{j_2}a=b^{j_2(k+1)}$ and 
\begin{equation*}
\begin{split}
\beta^{k+1}\alpha=(b^{j_2}a)^{k+1}a=b^{j_2\frac{k}{2}(k+2)}b^{j_2}aa=b^{j_2\frac{k}{2}(k+2)}b^{j_2}.
\end{split}
\end{equation*}
Thus, $\alpha \beta=\beta^{k+1}\alpha$ if and only if $b^{j_2(k+1)}=b^{j_2\frac{k}{2}(k+2)}b^{j_2}$, which is true if and only if $j_2k \equiv j_2\frac{k}{2}(k+2) \pmod {2k}$. Since $\gcd(j_2,2k)=1$, $j_2k \equiv j_2\frac{k}{2}(k+2) \pmod {2k}$ if and only if $k \equiv k(\frac{k}{2}+1) \pmod {2k}$, which is true as $k=2^{\ell}\geq 4$. One can adopt the similar arguments for the remaining choices of $i_1,i_2,j_1$ and $j_2$. Combining the above arguments, we conclude that $G=\langle \alpha, \beta \rangle$ and thus, $g$ is a permutation. It remains to show that $g \in \mathcal{N}(i_1,j_1,i_2,j_2)$, that is, 
\begin{equation}\label{e46}
gag^{-1}=\alpha,
\end{equation}
 and 
\begin{equation}\label{e47}
gbg^{-1}=\beta.
\end{equation}
Here, we verify these equations for the elements of the form $g(c_{2r})$, where $0 \leq r \leq \frac{k}{2}-1$. Using the definition of $a$ in Lemma \ref{HK2}, it is easy to see that $gag^{-1}(g(c_{2r}))=g(c_{2r+1})$. Now, from \eqref{LE48}, $\alpha(g(c_{2r}))=\alpha(\beta^{2r}(c_t))=g(c_{2r+1})$. Hence $gag^{-1}(g(c_{2r}))=\alpha(g(c_{2r})).$


Next, $gbg^{-1}(g(c_{2r})=g(c_{2(2r+k-1)})$ as from Lemma \ref{HK2}, $b(c_{2r})=c_{2(2r+k-1)}$. Now, using \eqref{LE48}, $\beta(g(c_{2r})=\beta^{2r+1}(c_t)=g(c_{2(2r+k-1)}).$ Thus, $gbg^{-1}(g(c_{2r})=\beta(g(c_{2r})$.

The same line of reasoning extends the verification to all remaining elements. Hence $g \in \mathcal{N}(i_1,j_1,i_2,j_2)$.
From the construction of 
$g$, it is evident that its form is determined exclusively by the choice of $c_t$, moreover, distinct values of $c_t$ give rise to distinct functions $g$.
Therefore $|\mathcal{N}(i_1,j_1,i_2,j_2)|\geq q$.  Further by the similar argument as in Lemma \ref{Atmost_lemma_T31}, we have $|\mathcal{N}(i_1,j_1,i_2,j_2)|\leq q$. Consequently, $|\mathcal{N}(i_1,j_1,i_2,j_2)|=q$.
\end{proof} 
\begin{thm}\label{Enumeration_HK2}
Let $f_2$ be a permutation group polynomials corresponding to the group $G$ constructed in Lemma \ref{HK2} and $\mathcal{P}_{f_2}$ be the number of permutation group polynomials that are of the form as $f_2$. Then 
\[
\mathcal{P}_{f_2}=\frac{(q!)^2}{5q\phi(\frac{q}{2})}.
\]
\end{thm}
\begin{proof}
Applying the similar argument as in Theorem \ref{Enumeration_T31}, we have 
\[
\mathcal{P}_{f_2}=\frac{(q!)^2}{|\mathfrak{N}_{\mathfrak{S}_q}(G)|}, 
\]
where $\mathfrak{N}_{\mathfrak{S}_q}(G)$ is the normalizer of $G$ in $\mathfrak{S}_q$. We also have 
\[
|\mathfrak{N}_{\mathfrak{S}_q}(G)|=\sum_{\substack{i_1,i_2 \in \{0,1\} \\   j_1,j_2 \in \{0,1,\ldots,\frac{q}{2}-1\}}}|\mathcal{N}(i_1,j_1,i_2,j_2)|
\]
as $\mathcal{N}(i_1,j_1,i_2,j_2)$'s are disjoint, where  $\mathcal{N}(i_1,j_1,i_2,j_2)$ is defined in Lemma \ref{Lemma_HK2}. From Lemma \ref{Lemma_HK2}, we have
\[
|\mathfrak{N}_{\mathfrak{S}_q}(G)|=q|C|,
\] 
where $C=\{(i_1,j_1,i_2,j_2) \mid (i_1,j_1,i_2,j_2)=(1,0,i,j_2) \text{ or } (1,\frac{q}{4},i,j_2) \text{ or } (0,\frac{q}{4},1,j_2), \gcd(j_2,\frac{q}{2})=1, 0 \leq i \leq 1, 0 \leq j_2 \leq \frac{q}{2}-1\}$. It is easy to see that 
\[
|C|=2\phi(\frac{q}{2})+2\phi(\frac{q}{2})+\phi(\frac{q}{2})=5\phi(\frac{q}{2}),
\]
which completes the proof.
\end{proof}
\begin{prop}
Let $f_2$ be a permutation group polynomial constructed in Lemma \ref{HK2} and its corresponding permutation polynomial tuple is $\underline{\beta}_{f_2}=(\beta_0,\beta_1,\ldots,\beta_{q-1})\in \mathfrak{S}_q^{q}$. Then there are $q!$ permutation group polynomials equivalent to $f_2$.
\end{prop}
\begin{proof}
The proof follows by similar arguments as used in Proposition \ref{Cor_HK1}.
  \end{proof}
\begin{lem}
Let $q=4k$ and $G =\langle a,b \rangle$ be the group in Lemma \ref{HK3}, where $k=2^{\ell}$ for some integer $\ell \geq 2$. Moreover, let $i_1,i_2 \in \{0,1\}$, $j_1, j_2 \in \{0,1,\ldots, \frac{q}{2}-1\}$ and $\mathcal{N}(i_1,j_1,i_2,j_2)=\{h \in \mathfrak{S}_q \mid hah^{-1}=b^{j_1}a^{i_1}, hbh^{-1}=b^{j_2}a^{i_2}\}$. Then  $\mathcal{N}(i_1,j_1,i_2,j_2) \neq \emptyset$ if and only if $i_1=1$, $i_2=0$, $2\mid j_1$ and $\gcd(j_2,2)=1$. In this case, $|\mathcal{N}(i_1,j_1,i_2,j_2)|=q$.
\end{lem}
\begin{proof}
First, we assume that $\mathcal{N}(i_1,j_1,i_2,j_2) \neq \emptyset$. We can prove that $i_1=1$, $i_2=0$, $2\mid j_1$ and $\gcd(j_2,2)=1$ by using the similar techniques used in Lemma \ref{Lemma_HK2}. For the converse part, we can consider the following function.
\[
    \begin{cases}
       h(c_{2r})= \beta^{2r}(c_{d})& \text{ if } 0 \leq r \leq \frac{k}{2}-1,  \\
       h(c_{2r+k})= \beta^{2r+k}(c_{d})& \text{ if } 1 \leq r \leq \frac{k}{2}-1,  \\
       h(c_{2r+2k+1})= \beta^{2r+1}(c_d)& \text{ if } 0 \leq r \leq \frac{k-2}{2}, \\
       h(c_{2r+3k})= \beta^{k+2r+1}(c_{d})&  \text{ if }  0  \leq r \leq \frac{k-2}{2}, \\
       h(c_{k-1-2r})= \alpha \beta^{k-2-2r} (c_{d})&  \text{ if }  0 \leq r \leq \frac{k}{2}-1,\\
       h(c_{2k-1-2r})= \alpha \beta^{2k-2r-2} (c_{d})&  \text{ if }  0 \leq r \leq \frac{k}{2}-2,\\
       h(c_{3k-2-2r})= \alpha \beta^{k-1-2r}(c_d)& \text{ if } 0 \leq r \leq \frac{k-2}{2}, \\
        h(c_{4k-1-2r})= \alpha \beta^{2k-1-2r}(c_{d})&  \text{ if }  0  \leq r \leq \frac{k-2}{2}, \\
        h(c_{k+1})= \beta^{k}(c_{d}),\\
        h(c_{k})= \alpha \beta^{k}(c_{d}),\\ 
	\end{cases}
\]
where $\alpha=b^{j_1}a^{i_1}$ and $\beta=b^{j_2}a^{i_2}$.
 Then we can show that $h$ is a permutation and $h \in \mathcal{N}(i_1,j_1,i_2,j_2)$ using an similar argument to that in Lemma \ref{Lemma_HK2}. In the end, we will have $|\mathcal{N}(i_1,j_1,i_2,j_2)|=2^m$.
\end{proof}
\begin{thm}\label{Enumeration_HK3}
Let $f_3$ be a permutation group polynomial corresponding to the group constructed in Lemma \ref{HK3}. $\mathcal{P}_{f_3}$ denotes the number of permutation group polynomials, which are of the form $f_3$. Then
\[
\mathcal{P}_{f_3}=\dfrac{4(q!)^2}{q^2\phi(\frac{q}{2})}.
\]
\end{thm}
\begin{proof}
The proof proceeds along the similar lines as the proof of Theorem~\ref{Enumeration_HK2}.    
\end{proof}
\begin{prop}
Let $f_3$ be a permutation group polynomial constructed in Lemma \ref{HK3} and its corresponding permutation polynomial tuple is $\underline{\beta}_{f_3}=(\beta_0,\beta_1,\ldots,\beta_{q-1})\in \mathfrak{S}_q^{q}$. Then there are $q!$ permutation group polynomials equivalent to $f_3$.
\end{prop}
\begin{proof}
The proof follows by similar arguments as used in Proposition  \ref{Cor_HK1}.
  \end{proof}

\end{document}